\documentclass{amsart}
\usepackage{amssymb}
\usepackage{amscd}
\usepackage{amsmath}
\usepackage{enumerate}
\usepackage[dvips]{graphicx}
\newtheorem{theo}{Theorem}
\newtheorem{lema}[theo]{Lemma}
\newtheorem{cor}[theo]{Corollary}
\newtheorem{prop}[theo]{Proposition}

\newtheorem{definition}[theo]{Definition}

\newtheorem{remark}[theo]{Remark}

\newcommand{\AAA}{{\mathbb{A}}}
\newcommand{\CC}{{\mathbb{C}}}
\newcommand{\NN}{{\mathbb{N}}}
\newcommand{\PP}{{\mathbb{P}}}
\newcommand{\QQ}{{\mathbb{Q}}}

\newcommand{\KK}{{\mathbb{K}}}
\newcommand{\SSS}{{\mathbb{S}}}
\newcommand{\ZZ}{{\mathbb{Z}}}

\newcommand{\calA}{{\mathcal{A}}}
\newcommand{\calB}{{\mathcal{B}}}
\newcommand{\calC}{{\mathcal{C}}}
\newcommand{\calD}{{\mathcal{D}}}
\newcommand{\calE}{{\mathcal{E}}}

\newcommand{\calG}{{\mathcal{G}}}

\newcommand{\calI}{{\mathcal{I}}}
\newcommand{\calJ}{{\mathcal{J}}}

\newcommand{\calO}{{\mathcal{O}}}
\newcommand{\calP}{{\mathcal{P}}}
\newcommand{\calS}{{\mathcal{S}}}

\newcommand{\calX}{{\mathcal{X}}}

\newcommand{\calZ}{{\mathcal{Z}}}
\newcommand{\comp}{{\circ}}

\begin{document}
\title[Nash Problem]{Nash problem for surface singularities is a topological problem}
\author{Javier Fern\'andez de Bobadilla}
\address{ICMAT. CSIC-Complutense-Aut\'onoma-Carlos III}
\email{javier@mat.csic.es}
\thanks{Research partially supported by the ERC Starting Grant project TGASS and by Spanish Contract MTM2007-67908-C02-02. The author thanks to the Faculty de Ciencias Matem\'aticas of the Universidad Complutense de Madrid for excellent working conditions.}
\date{15-1-2008}
\subjclass[2000]{Primary: 14B05, 14J17, 14E15, 32S05, 32S25, 32S45}
\begin{abstract}
We address Nash problem for surface singularities using {\em wedges}. 
We give a refinement of the characterisation~in~\cite{Re} of the image of the Nash map in terms of wedges. Our improvement consists in
a characterisation of the bijectivity of the Nash mapping 
using wedges defined over the base field, which are convergent if the base field is $\CC$, and whose generic arc has transverse lifting 
to the exceptional divisor.
This improves the results of M. Lejeune-Jalabert and A. Reguera~\cite{LR} for the surface case. 
In the way to do this we find a reformulation of Nash problem in terms of branched covers of normal surface singularities.
As a corollary of this reformulation we prove that the image of the Nash mapping is characterised by the combinatorics of a
resolution of the singularity, or, 
what is the same, by the topology of the abstract link of the singularity in the complex analytic case.
Using these results we prove several reductions of the Nash problem,
the most notable being that, if Nash problem is true for singularities having rational homology sphere links, then it is true in
general. 
\end{abstract}


\maketitle

\section{Introduction}

Nash problem~\cite{Na} was formulated in the sixties (but published later) in the attempt to understand the relation between the 
structure of resolution of singularities of an algebraic variety $X$ over a field of characteristic $0$ and the space of arcs (germs of algebroid curves) in the variety.
He proved that the space of arcs centred at the singular locus (endowed with a infinite-dimensional algebraic variety structure) 
has finitely many irreducible components, and proposed to study the relation of these components with the 
essential irreducible components of the exceptional set a resolution of singularities. An irreducible component $E$ of the exceptional divisor of a
resolution of singularities 
\[\pi:\tilde{X}\to X\]
is called essential,
if given any other resolution 
\[\pi':\tilde{X}'\to X\]
the birational transform of $E$ to $\tilde{X}'$ is an irreducible component of the exceptional divisor. Nash defined a mapping from the set of irreducible components of the
space of arcs centred at the singular locus to the set of essential components of a resolution as follows: he assigns to each component $Z$ 
of the space of arcs centred at the singular locus the unique component of the exceptional divisor
which meets the lifting of a generic arc of $Z$ to the resolution. Nash established the injectivity of this mapping and asked whether it is bijective. He viewed as a plausible fact that Nash mapping is bijective in the surface case, and
also proposed to study the higher dimensional case.

Nash gave an affirmative answer to his problem in the case of $A_k$-singularities. Since then, there has been much progress showing an affirmative answer to the problem
for many classes of singularities: toric singularities of arbitrary dimension, quasi-ordinary singularities, certain infinite families of non-normal 
threefolds, minimal surface singularities, sandwiched surface singularities, quotient surface singularities, and other classes of surface singularities
 defined in terms of the combinatorics of the minimal resolution 
(see \cite{Go},\cite{IK},\cite{I1},\cite{I2},\cite{LR1},\cite{Mo},\cite{Pe},\cite{Pet},\cite{Pl},\cite{PlSp}, \cite{PP1},\cite{PP2},\cite{Re1},\cite{Re2}). 
However, Ishii and Kollar showed in~\cite{IK} a 4-dimensional example with non-bijective Nash mapping. Now the general problem
has turned into characterising the class of singularities with bijective Nash mapping. Besides Nash problem, the study of arc spaces is interesting because it lays the
foundations for motivic integration and because the study of its geometric properties reveals properties of the underlying varieties (see papers of Denef, Loeser, de Fernex, Ein, Ishii, Lazarsfeld, Mustata, Yasuda and others).

Nash problem seems different in nature in the surface case than in the higher dimensional case,
since birational geometry in dimension $2$ is much easier than in higher dimension.
For example the essential components are the irreducible components of the exceptional divisor of a minimal resolution of singularities.
Although Nash problem in known for many classes of surfaces it is not yet known in general for the surface case. 
Even for the case
of the rational double points $E_6$, $E_7$ and $E_8$ the proof has been obtained only very recently: see~\cite{PlSp} for the $E_6$
case and~\cite{Pe} for the quotient surface singularities, which includes the rational double points and uses essentially
the developments of this paper. 
From now on we shall concentrate in the surface case, we let $(X,O)$ be a normal surface singularity defined over a 
field of characteristic $0$, and $\calX_\infty$ denotes the space of arcs through the singular point.

Let us explain the approach to Nash problem based on wedges, due to M. Lejeune-Jalabert~\cite{Le}.
Let $E_u$ be an essential component of a surface singularity $(X,O)$. 
Denote by $N_{E_u}$ the set of arcs whose lifting meets $E_u$. The space of arcs centred at the singular point splits as the union of the $N_{E_u}$'s.
It is known (Remark 2.3 of~\cite{Re}) that $E_v$ is {\em not} in the image of the Nash map if
and only if $N_{E_v}$ is in the Zariski closure of $N_{E_u}$ for a different essential component $E_u$. If Curve Selection Lemma were true in $\calX_\infty$ then, for any arc 
\[\gamma:Spec(\KK[[t]])\to (X,O)\]
in $N_{E_v}$ there should exists a curve in $\calX_\infty$ with special point $\gamma$ and generic point an arc on $N_{E_u}$. To give a curve in $\calX_\infty$ amounts
to give a morphism
\[\alpha:Spec(\KK[[t,s]])\to (X,O)\]
mapping $V(t)$ to $O$ (we see it as a ``family of arcs parametrised by s'').
Such a morphism is called a {\em wedge}. The lifting to $\tilde{X}$ of a generic arc in $N_{E_v}$ is transverse to $E_v$, and if a wedge $\alpha$ 
has special arc equal to $\gamma$ and generic arc in $N_{E_u}$ it is clear that the rational lifting 
\[\pi^{-1}\comp\alpha:\KK[[t,s]])\to\tilde{X}\]
has an indetermination point at the origin, and hence there is no morphism lifting $\alpha$ to $\tilde{X}$. In~\cite{Le}, M. Lejeune-Jalabert proposes to attack Nash
problem by study
ing the problem of lifting wedges whose special arc is a transverse arc through an essential component of $(X,O)$.

Since Curve Selection Lemma in not known to hold in the space of arcs A. Reguera~\cite{Re} introduced $K$-wedges, which are wedges 
\[\alpha:Spec(K[[t,s]])\to (X,O)\]
defined over a field extension $K$ of
$\KK$ and proved the following characterisation: an essential component $E_v$ is in the image of the Nash map if and only 
if any wedge whose special arc is {\em the generic point} of $N_{E_v}$ and whose generic arc is centered at the singular point $O$ 
of $X$, admits a lifting to the resolution.
However the field of definition $K$ of the involved wedges has infinite transcendence degree over $K$ and, hence,
it is not easy to work with them.
Building on this result, A. Reguera~\cite{Re} and M. Lejeune-Jalabert~(Proposition~2.9,~\cite{LR})
proved a sufficient condition for a
divisor $E_v$ to be in the image of the Nash map based on wedges defined over the base field: it is enough to check that any $\KK$-wedge whose special arc is transverse
to $E_i$ arc through a very dense collection of closed points of $E_i$ lifts to $\tilde{X}$ (a very dense set is a set which intersects any countable intersection of 
dense open subsets). The results of~\cite{Re}~and~\cite{LR} hold in any dimension.

We say that a component $E_u$ of the exceptional divisor is adjacent to $E_v$ if $N_{E_v}$ is contained in the Zariski closure of $N_{E_u}$.
The first main result of this article is a characterisation of all the possible adjacencies between essential components of the exceptional divisor of a resolution 
in terms of wedges defined over the base field. We are even able to show that the generic arc of the wedge satisfies certain genericity conditions. 
For example we can show that the generic arc is transverse to $E_u$.

We denote by $\Delta_{E_u}$ the set of arcs in $N_{E_u}$ whose lifting is not transverse to $E_u$ through a non-singular point of $E$.
We prove (see Theorem~\ref{Nashwedge1} for a more precise version):

\textbf{Theorem A.}
{\em Let $(X,O)$ be a normal surface singularity defined over an uncountable algebraically closed field $\KK$ of characteristic $0$.
Let $E_u$, $E_v$ be different essential irreducible components of the exceptional divisor $E$ of a resolution. 
We consider $E$ with its reduced structure. Let $\calZ$ be a proper Zariski closed subset of $N_{E_u}$. Equivalent are:
\begin{enumerate}
\item the component $E_u$ is adjacent to $E_v$.
\item There exists a $\KK$-wedge whose special arc has transverse lifting to $E_v$ at a non-singular point of $E$ and with generic arc belonging to $N_{E_u}$.
\item There exists a $\KK$-wedge whose special arc has transverse lifting to $E_v$ at a non-singular point of $E$ and with generic arc belonging to $N_{E_u}\setminus\calZ$.
\end{enumerate}
If the base field is $\CC$ the following conditions are also equivalent:
\begin{enumerate}[(a)]
\item Given {\em any} convergent arc $\gamma$ whose lifting is transverse to $E_v$ at a non-singular point of $E$ there exists a convergent $\CC$-wedge with special arc $\gamma$ and generic arc 
belonging to $N_{E_u}$.
\item Given {\em any} convergent arc $\gamma$ whose lifting is transverse to $E_v$ at a non-singular point of $E$ there exists a convergent $\CC$-wedge with special arc $\gamma$ and generic arc 
having lifting to the resolution transverse to $E_u$ through a non-singular point of $E$.
\end {enumerate}}

An immediate Corollary characterises the image of the Nash maps in terms of $\KK$-wedges:\\

\textbf{Corollary B.}
{\em Let $(X,O)$ be a normal surface singularity defined over an uncountable algebraically closed field $\KK$ of characteristic $0$.
Let $E_v$ be an essential irreducible component of the exceptional divisor. Equivalent are:
\begin{enumerate}
\item The component $E_v$ is in the image of the Nash map.
\item It does not exist a different component $E_u$ and a $\KK$-wedge whose special arc has transverse lifting to $E_v$ at a non-singular point of $E$ 
and with generic arc belonging to $N_{E_u}$.
\end{enumerate}
If the base field is $\CC$ the following condition is also equivalent:
\begin{enumerate}[(a)]
\item There exists a convergent arc $\gamma$ whose lifting is transverse to $E_v$ at a non-singular point of $E$ 
such that there is no convergent $\CC$-wedge whose special arc is $\gamma$ and whose generic arc has a lifting to the resolution transverse to $E_u$ 
through a non-singular point of $E$, for a different component $E_u$ of the exceptional divisor.
\end {enumerate}}

Our result is based on the Curve Selection Lemma of A. Reguera (see Theorem~4.1 and Corollary~4.8 of~\cite{Re}), and improves the 
statement corresponding to surfaces of Corollary 2.5~of~\cite{LR} in the following sense:
\begin{itemize}
\item in order to prove that $E_v$ is not in the image of the Nash map it is sufficient to exhibit a {\em single}
wedge defined over the base field realising an adjacency.
\item If $\KK=\CC$, in 
order to prove that $E_v$ is in the image of the Nash it is sufficient to find a {\em single} convergent arc which can not be the
special arc of a wedge realising an adjacency. This has turned out to be
essential for the proof of the bijectivity of the Nash mapping for quotient surface singularities~\cite{Pe}.
\item The convergence and the condition that the wedge avoids a proper closed subset $\calZ$ of $N_{E_u}$ is also very useful
in practice since it
allows to work geometrically with wedges whose generic arc has also a tranverse lifting. See~\cite{Pe} for an application of these
ideas. 
\end{itemize}

Our condition on the wedges is more precise that the liftability, since, when a component $E_v$ is
not in the image of the Nash map, we want to keep track of the responsible adjacencies.
However, we prove also an improvement of the result of~\cite{LR} in terms of the original condition of lifting wedges of~\cite{Le}:\\

\textbf{Theorem C.}
{\em Let $(X,O)$ be a normal surface singularity defined over $\CC$. Let $E_v$ be any essential irreducible component of the exceptional divisor of a resolution of
singularities.
If there exists a convergent arc $\gamma$ whose lifting is transverse to $E_v$ such that
any $\CC$-wedge having $\gamma$ as special arc lifts to $\tilde{X}$, then the component $E_v$ is in the image of the Nash map.}\\

The ideas of the proofs and the plan of the paper are as follows: 

If the Zariski closure of $N_{E_u}$ contains $N_{E_v}$ we can use
Corollary~4.8~of~\cite{Re} to obtain a $K$-wedge (with $K$ an infinite transcendence degree extension of $\KK$)
whose special arc is {\em the generic point} of $N_{E_v}$ and whose
generic arc lifts to $E_u$. After this we can follow a specialisation procedure similar to the one in~\cite{LR} 
to obtain a $\KK$-wedge
whose special arc has transverse lifting to $E_v$ and with generic arc belonging to $N_{E_u}$. What we will actually 
do is to produce a slight improvement of Corollary~4.8~of~\cite{Re} 
which, after the above mentioned specialisation procedure gives 
$\KK$-wedges with the properties required in Theorem A, (2) and (3). The improvement of Corollary~4.8~of~\cite{Re} is needed because
applying it directly we do not obtain wedges avoiding the proper closed subset $\calZ$ of $N_{E_u}$. This is done 
in Section~\ref{exis}.

Most of the technique introduced in this paper is devoted to the proof of the converse statement. Let
$\gamma$ be an arc whose lifting to $\tilde{X}$ is transverse to $E_v$. A 
$\KK$-wedge whose special arc is $\gamma$ hand with generic arc belonging to $N_{E_u}$ (with $E_u$ another component of the
exceptional divisor) will be called a wedge realising an adjacency from $E_u$ to $\gamma$.

In Section~\ref{appwdg} we use an Approximation Theorem to replace a $\KK$-wedge realising an adjacency 
by an algebraic one with the same property with respect to another transverse arc $\gamma'$, and retaining 
genericity conditions. 

In Section~\ref{awfm}, using Stein Factorisation, we ``compactify'' the wedge and factorise it through a finite covering of 
normal surface singularities ``realising an adjacency from $E_u$ to $\gamma'$'' (see~Definitions~\ref{branchadj}~and~\ref{branchadj2}). 
We prove that there exists a wedge 
realising an adjacency from $E_u$ to $\gamma'$ if and only if there exists a finite covering realising an adjacency from $E_u$ to $\gamma'$.

In Section~\ref{mw} we work in the complex analytic category. We use a topological argument and a suitable change of complex structures we prove that given two convergent 
arcs $\gamma$ and $\gamma'$ on a complex analytic normal surface singularities there exists a finite covering
realising an adjacency from $E_u$ to $\gamma$ if and only
if there exists a finite covering realising an adjacency from $E_u$ to $\gamma'$. This technique allows to move wedges in a very flexible way, and is the key to the
following surprising result (see Theorem~\ref{combinatorio}, and Corollary~\ref{topo}):\\

\textbf{Theorem D.}
{\em The set of adjacencies between exceptional divisors of a normal surface singularity is a combinatorial property
of the singularity: it only depends on the dual weighted graph of the minimal good resolution. In the complex analytic case this means that the set of adjacencies only depends
on the topological type of the singularity, and not on the complex structure}.\\

In Section~\ref{principal} we use Lefschetz Principle to reduce the existence of a finite covering realising an adjacency from $E_u$ to a $\gamma'$ to the analogue statement in the complex analytic case. 
This allows to finish the proof the main results of the paper. 

Finally, in Section~\ref{aplicaciones} we use the fact that Nash problem is topological to study and compare the adjacency structure of different singularities:
we prove reductions of Nash problem for singularities with symmetries in the dual weighted graph of the minimal good resolution
(see~Proposition~\ref{simmetries} and Corollary~\ref{simmetries2}). We prove results comparing
the adjacency structure of different singularities (see~Corollary~\ref{comparacion1}). 
We also reduce the Nash problem in the following
sense: we introduce {\em extremal graphs}, which is a subclass of the class of dual graphs with only rational vertices and no loops (see Definition~\ref{extremal}), 
and {\em extremal rational homology spheres}, which are the plumbing $3$-manifolds associated with extremal graphs.\\

\textbf{Corollary E.}
{\em If the Nash mapping is bijective for normal surface singularities whose minimal good resolution graph is extremal
then it is bijective in general. Equivalently, if the Nash mapping is bijective for all complex analytic normal
surface singularities having extremal $\QQ$-homology sphere links then it is bijective for any normal surface singularity.}\\

The last Corollary improves Proposition~4.2~of~\cite{LR}, which reduces Nash problem for surfaces to the class of surfaces having only rational vertices in its resolution,
and makes essential use of Theorem D.

The author thanks to M. Lejeune-Jalabert for a conversation in which he learnt on the possible relation between Nash and wedge
problems. He also thanks useful comments from P. Popescu-Pampu, H. Hauser, M. Pe Pereira, G. Rond and the referee.

\section{Terminology}

Let $\KK$ be an algebraically closed field of characteristic $0$. Let $X$ be a normal algebraic surface over $\KK$ with a singularity at a point $O$.
Since we will be interested in the germ $(X,O)$ we may assume $X$ to be embedded in $\KK^N$, being $O$ the origin.
Denote by $I$ the defining ideal of $X$ and let $(F_1,....,F_r)$ be generators of $I$. If the coordinates of $\KK^N$ are $x_1,...,x_N$ then each 
$F_i$ is a polynomial in these variables.

Unless we state the contrary we denote by $\pi:\tilde X\to X$ an Hironaka resolution of the singularity of $X$ at $O$ (obtained by a blowing up at an ideal whose zero set is the 
singular locus~\cite{Hi1}). Let $E=\bigcup_{u=1}^sE_u$ be the decomposition 
in irreducible components of $E$.

Given a field $K$ containing the base field $\KK$, a $K$-arc in $(X,O)$ is a scheme morphism
\[\gamma:Spec(K[[t]])\to (X,O).\]
It is determined by the formal power series $x_i\comp\gamma(t)\in K[[t]]$ for $1\leq i\leq n$ (the coordinate series).
A $\KK$-arc is algebraic if its coordinate
series are algebraic power series: there exists a polynomial $P_i\in\KK[t,x]$ such that $P_i(t,x_i\comp\gamma(t))=0$ for any $i\leq N$. In the case that $\KK=K=\CC$, a $\CC$-arc is convergent if its coordinate series are convergent.

Denote by $\calX_\infty$ be arc space of $(X,O)$, with its infinite-dimensional scheme structure. There is a natural
bijection between the set of $K$-arcs and the $K$-valued points of $\calX_\infty$. 
See~\cite{IK} for a more detailed exposition on arc spaces.

Denote by $N_{E_u}$ the subspace of formal arcs whose lifting to $\tilde{X}$ sends the special point into $E_u$; these arcs are called {\em arcs through $E_u$}. Let $\overline{N}_{E_u}$ be its Zariski closure in $\calX_\infty$ and $\dot{N}_{E_u}$ the subset of arcs in $N_{E_u}$ whose lifting meets $E_u$ transversely at a non singular point of $E$; these arcs are called {\em transverse arcs through $E_u$}.

A $K$-wedge in $(X,O)$ is a scheme morphism
\[\alpha:Spec(K[[t,s]])\to (X,O)\]
which sends the point $V(t)$ to $O$. Its coordinate series are formal power series in the variables $t,s$. A $\KK$-wedge is said to be algebraic if 
all the coordinate series are algebraic (that is, satisfy a polynomial equation $P_i(t,s,x_i\comp\alpha(t,s))=0$ for $P_i\in\KK[t,s,x]$). 
In the case that $\KK=K=\CC$, a $\CC$-wedge is convergent if its coordinate series are convergent.

Given a $K$-wedge, we call the $K$-arc $\alpha_s(t):=\alpha(t,0)$ the {\em special} $K$-arc associated to it, and the $K((s))$-arc $\alpha_g$
(same coordinate series, but viewed in $K((s))[[t]]$) the {\em generic} $K((s))$ arc associated to it.

As the space of arcs, the space of wedges $\calX^{Sing}_{\infty,\infty}$ as above has a structure of infinite dimensional algebraic variety.
There is a natural bijection between the set of $K$-wedges and the $K$-valued points of $\calX^{Sing}_{\infty,\infty}$ (see~~\cite{LR}).

Since in this article we are interested in Nash Problem for a normal surface (and hence isolated) singularity, we only need to use wedges such that all both the generic 
and the special arc are centred at the singular point. In more general contexts (see~\cite{LR}) wedges with the general arc not centred at the singular locus are
considered.

\begin{definition}
A $K$-wedge realises an adjacency from $E_u$ to $E_v$ if its generic arc belongs to $N_{E_u}$ and its special arc
belongs to $\dot{N}_{E_v}$ (it is transverse to $E_v$).
\end{definition}

We will be able to produce wedges for which the generic arc satisfies any finite amount of genericity conditions. This 
can be formulated as follows:

\begin{definition}
Let $\calZ$ be a proper closed subset of $\overline{N}_{E_u}$. A $K$-wedge realises an adjacency from $E_u$ to $E_v$
avoiding $\calZ$ if it realises an adjacency from $E_u$ to $E_v$ and its generic arc does not belong to $\calZ$.
\end{definition}

A very natural class of wedges are those realising an adjacency and
for which the generic arc is transverse to $E_u$. For this we define the
proper closed subset $\Delta_{E_u}\subset\overline{N}_{E_u}$ to be the Zariski closure of the complement of $N_{E_u}\setminus\dot{N}_{E_u}$. Any wedge realising an adjacency from $E_u$ to $E_v$ and avoiding $\Delta_{E_u}$ has the
desired properties.

\begin{remark}
\label{lifting1}
A $\KK$-wedge $\alpha$ realising an adjacency from $E_u$ to $E_v$ is a wedge with special arc in $\dot{N}_{E_v}$ which does not lift to $\tilde{X}$ (that is, the rational
map $\pi^{-1}\comp\alpha$ is not defined at the origin of $Spec(\KK[[t,s]])$). This makes the link with the original wedge problem of~\cite{Le}.
\end{remark}

\section{Existence of formal wedges over the base field}
\label{exis}

The aim of this section is to prove that $(1)$ implies $(3)$ in Theorem A. We follow closely the method of~\cite{Re}~and~\cite{LR}, which consists in
two steps: finding a wedge defined over a field extension $\KK\subset K$, and perform an specialisation procedure to deduce the existence
of wedges over the base field. In~\cite{Re}~and~\cite{LR} wedges that do not lift to the resolution are produced. Our aim is to produce wedges realising
adjacencies avoiding a closed subset $\calZ$ in $\overline{N}_{E_u}$. Both conditions are related, but are not the same: a wedge realising an adjacency
does not lift to the resolution, but a wedge which does not lift to the resolution does not necessarily realises an adjacency. The condition of avoiding
a closed subset $\calZ$ is not considered in~\cite{Re}~and~\cite{LR}.

\subsection{An improved Curve Selection Lemma}

Let $E_v$ be an essential divisor. Suppose that there is an adjacency from $E_u$ to $E_v$. Let $z$ be the generic point corresponding
to the irreducible closed
subset $N_{E_v}\subset\calX_\infty$, and $k_z$ its residue field. Corollary~4.8~of~\cite{Re} is a Curve Selection Lemma
in arc spaces which implies (see the proof of Theorem~5.1~of~\cite{Re}) 
that there exists a finite extension $K$ of $k_z$ and a $K$-wedge whose special arc is the generic point of $N_{E_v}$ and whose generic arc belongs to $N_{E_u}$. 
Here we improve the Curve Selection Lemma so we can show that the generic arc of the wedge also satisfy certain genericity conditions.
The proof we give here is a modification
of the proof of Corollary~4.8~of~\cite{Re}, and is a Corollary of Theorem~4.1~of~\cite{Re} as well. However
we want to make several points more explicit for later use and we also want to prove the
existence of wedges avoiding a proper closed subset $\calZ$ of $N_{E_u}$.
We shall follow closely~\cite{Re}. 

\begin{definition}[\cite{Re}]
An irreducible Zariski-closed subset $N$ of $\calX_\infty$ is generically stable if there exists an affine open subscheme $U$ of $\calX_\infty$ such that $N\cap U$ is 
non-empty and its defining ideal is the radical of a finitely generated ideal.
\end{definition}

Let $N$ be any irreducible closed subset of $\calX_\infty$, denote by $z$ its scheme-theoretic
generic point and $k_z$ its residue field. Notice that $\calX_\infty$ is an affine scheme. Let $R$ be its coordinate ring. 
There is a prime ideal $I\subset R$ defining the set $N$. Then $k_z$ is the quotient field of $R/I$ and the point $z$ is 
the $k_z$-valued point
\[Spec(k_z)\to \calX_\infty\] 
associated to the natural homomorphism given by the composition
\[R\to R/I\hookrightarrow k_z.\]

Recall that the arc space $\calX_\infty$ is the scheme representing the contravariant functor which assigns to any affine scheme $Spec(A)$ the set of 
$A$-arcs
\[\gamma:Spec(A[[t]])\to X\]
which send the closed subset $V(t)$ to the origin $O$ of $X$. By the above discussion the generic point $z$ of $N$ is identified with a 
$k_z$-arc in $\calX_\infty$. Similarly giving a $K$-wedge in $X$ is the same than giving a morphism
\[\alpha:Spec(K[[s]])\to\calX_\infty.\]

Here is our improved Curve Selection Lemma, whose proof is an adaptation of the proof of Corollary~4.8~of~\cite{Re}:

\begin{lema}
\label{curveselectionlemma}
Let $N$ and $N'$ be two irreducible closed subsets of $\calX_\infty$ such that $N$ is generically stable,  contained in $N'$ and different to it. 
Let $\calZ$ be any other closed subset of $\calX_\infty$ such that $N'$ is not contained in it. Let $z$ be the generic point of $N$ and 
$k_z$ its residue field. There exists a finite field extension $k_z\subset K$ and $K$ wedge whose special arc is the generic point $z$ and whose 
generic arc belongs to $N'\setminus\calZ$. 
\end{lema}
\begin{proof}
By the discussion preceding the Lemma finding such a wedge is equivalent to finding finite extension $k_z\subset K$ and a morphism
\begin{equation}
\label{wedgegengen}
\alpha:Spec(K[[s]])\to N'
\end{equation}
such that the image of the generic point does not fall in $\calZ$ and, if $\calX_\infty=Spec(R)$ and if $I$ is the ideal defining $N$ in
$\calX_\infty$, the restriction of $\alpha$ to the closed point of $Spec(K[[s]])$ is associated to the composition of the natural ring 
homomorphisms
\begin{equation}
R\to R/I\hookrightarrow k_z\hookrightarrow K.
\end{equation}

By Corollary~4.6~(i)~of~\cite{Re}, since $N$ is generically stable the completion $\hat{\calO}_{N',z}$ of the 
local ring $\calO_{N',z}$ at the maximal ideal is a Noetherian ring of dimension at least $1$ (the inclusion $N\subset N'$ is strict).
The ring $\hat{\calO}_{N',z}$ is equicharacteristic since it contains the base field $\KK$. Since it is complete, by Theorem 28.3~of~\cite{Ma}
it has a coefficient ring $k_z$ and therefore is a quotient $k_z[[X_1,...,X_r]]/\calI$, with $r$ the minimal number of generators of the maximal ideal,
and $\calI$ an ideal in $k_z[[X_1,...,X_r]]$. 

Let $J$ the ideal of $\calZ\cap N'$ in the coordinate ring of $\calO_{N',z}$. Denote by $\hat{J}$ its completion to $\hat{\calO}_{N',z}$.
Let $d$ be the dimension of $\hat{\calO}_{N',z}$. If $d=1$ then $\sqrt{\hat{J}}$ equals the maximal ideal of 
$\hat{\calO}_{N',z}$ because otherwise $\calZ$ contains $N'$.
If $d>1$ and $\sqrt{\hat{J}}$ is different to the maximal ideal then, by Prime Avoidance (see Lemma~3.3~of~\cite{Ei}), there exists an element $b_1$ in $\hat{\calO}_{N',z}$ and not belonging to any minimal prime containing 
$\hat{J}$, such that the dimension
of $\calO_{N',z}/(b_1)$ equals $d-1$. Define $\hat{J}_1=\hat{J}+(b_1)$. By construction no irreducible component of
$Spec(\calO_{N',z}/(b_1))$ is contained in $V(\hat{J}_1)$ and hence the dimension of $V(\hat{J}_1)$ is at most $d-2$. We proceed inductively and construct
$b_1,...,b_{d-1}$ such that the dimension of $Spec(\calO_{N',z}/(b_1,...,b_{d-1}))$ equals $1$ and the dimension of $V(\hat{J}_{d-1})$ is at most $0$, where 
$\hat{J}_{d-1}$ equals $\hat{J}+(b_1,...b_{d-1})$.

Since $\calO_{N',z}/(b_1,...,b_{d-1})$ is the quotient $k_z[[X_1,...,X_r]]/\calI+(b_1,...,b_{d-1})$ we have that $\calC=Spec(\calO_{N',z}/(b_1,...,b_{d-1}))$ is an
algebroid curve over $k_z$, and since $V(\hat{J}_{d-1})$ is at most $0$ dimensional it contains at most the closed point of the curve. Let $\overline{k_z}$ be the
algebraic closure of $k_z$. In the ring $\overline{k_z}[[X_1,...,X_r]]$ the ideal $\sqrt{\calI+(b_1,...,b_{d-1})}$ can be expressed as the intersection of finitely
many prime ideals $\calP_1,...,\calP_m$, each of them corresponding to a geometrically irreducible component $\calC_i$ of the curve $\calC$. By Noetherianity each of
curves $\calC_i$ is defined over a finite extension $K$ of $k_z$.

The desired morphism $\alpha$ may be taken as the composition of the normalisation mapping
\[\alpha:Spec(K[[s]])\to\calC_1\]
with the natural morphism
\[\calC_1\to N'.\]
\end{proof}

\subsection{Locally closed subsets of wedges realising adjacencies}
\label{spewedges}
Having a $K$-wedge realising an adjecency from $E_u$ to $E_v$ avoiding $\calZ$ our aim is to find a wedge with the same properties, 
but defined over the base field. Here we follow~\cite{LR}. The first step is to prove the following analog of Proposition~2.5~in~\cite{LR}. The way of
proving the proposition prepares the way to an approximation results for wedges needed in the nest section. 

\begin{prop}
\label{locallyclosed}
Let $E_u$ and $E_v$ be two essential divisors of a resolution of $X$ and $\calZ$ a proper closed subset of $\overline{N}_{E_u}$. If there is a $K$-wedge $\alpha$
realising an adjacency from $E_u$ to $E_v$ avoiding $\calZ$, and such that the image of the special point of the 
special arc $\alpha_s$ is the generic point of $E_u$, then there is a locally closed subset $\Lambda$ of
the space of wedges $\calX_{\infty,\infty}^{Sing}$ such that the $K$-point associated to $\alpha$ belongs to it
and, for any field extension $\KK\subset L$, any $L$-point of $\Lambda$ corresponds to a $L$-wedge
realising an adjacency from $E_u$ to $E_v$ avoiding $\calZ$.
\end{prop}

Some discussion is needed before proving the Proposition.

\subsubsection{Power series expansions of arcs with transverse lifting}
\label{subsubarcos}
Here we give finitely many polynomial conditions ensuring than an arc is in $N_{E_u}$ or in $\dot{N}_{E_u}$ for a 
certain component $E_u$ of the exceptional divisor. The discussion is similar to that of Proposition~3.8~of~\cite{Re}.

By Main Theorem I~of~\cite{Hi1} there exists an ideal $\calJ$ in the coordinate ring $\KK[X]$ of $X$ such that the resolution $\pi$ is the blow up with
respect to $\calJ$ and such that its zero set is the origin. Consider polynomials $g_0,...,g_s$ in $x_1,...,x_N$ which generate $\calJ$.
Let $\tilde{X}$ be the Zariski closure in $X\times\PP^s$ of the graph of the mapping 
\[(g_0:...:g_s):X\setminus\{O\}\to\PP^s.\]
The morphism $\pi:\tilde{X}\to X$ coincides with the restriction to $\tilde{X}$ of the projection of $X\times\PP^s$ to the first factor. Consequently,
the exceptional divisor $E$ is an algebraic subset of $\PP^s$. Denote by $J_u$ the homogeneous ideal associated to the component $E_u$; 
let $h_{u,1},...,h_{u,k_u}$ be homogeneous polynomials generating $J_u$.

Consider indeterminates $A_{i,j}$, where $i\in\{1,...,N\}$ and $j\in\NN$, for each $i\leq N$ we consider the series 
\[Y_i(t):=\sum_{j\in\NN}A_{i,j}t^j\]
inside the ring $\KK[A_{i,j}][[t]]$. The tuple 
$$\Gamma(t)=(Y_1(t),...,Y_N(t))$$
is a $\KK[A_{i,j}]$-point in the arc space of $\KK^N$ which we call the universal arc. The coordinates of any $K$-arc are obtained by substituting the variables $A_{i,j}$ by values $a_{i,j}\in K$.

For any $l\in\{0,...,s\}$, if we consider the power series expansion
\[g_l(Y_1(t),...,Y_n(t)):=\sum_{k\in\NN}Q_{l,k}t^k,\]
each coefficient $Q_{l,k}$ is a polynomial in $\KK[A_{i,j}]$, for $i\in\{1,...,N\}$ and $j\leq k$.

\begin{definition}
\label{1char}
Given any $K$-arc 
\[\gamma(t)=(...,\sum_{j\in\NN}a_{i,j}t^j,...)\]
we define the first order of $\gamma$ with respect to $\pi$ to be the minimum $m(\gamma)$ 
of the orders in $t$ of the power series
$g_i(\gamma(t))$ for $i\in\{0,...,s\}$.
\end{definition}

Consider the system $S_1(M)$ of polynomial equations:
\begin{equation}
\label{ordereqn}
Q_{l,k}(A_{i,j})=0\quad\text{for}\quad l\in\{0,...,s\}\quad\text{and}\quad k<M.
\end{equation}
By definition of $m(\gamma)$, the coefficients $a_{i,j}$ of the coordinate
series of $\gamma$ satisfy the system $S_1(M)$ for $M=m(\gamma)$, but not for $M=m(\gamma)+1$. This leads to the following:

\begin{remark}
\label{uppersem}
The first order of $\gamma$ with respect to $\pi$ is upper semicontinuous in the Zariski topology of $\calX_\infty$.
\end{remark}

Given any arc $\gamma\in\calX_\infty$, the lifting
\[\tilde{\gamma}:Spec(K[[t]])\to\tilde{X}\]
of $\gamma$ to $\tilde{X}$ has the coordinate expansion
\begin{equation}
\label{lifting}
\tilde{\gamma}(t):=((x_1(\gamma(t)),...,x_N(\gamma(t)),(g_0(\gamma(t)):...:g_s(\gamma(t)))\in\CC^N\times\PP^s.
\end{equation}

Dividing each $g_i(\gamma(t))$ by $t^{m(\gamma)}$, we can evaluate $\tilde{\gamma}$ at $0$ and obtain:
\begin{equation}
\label{specialpoint}
\tilde{\gamma}(0)=((0,...,0),(Q_{0,m(\gamma)}(a_{i,j}):...:Q_{s,m(\gamma)}(a_{i,j})),
\end{equation}
where at least one of the coordinates $Q_{l,m(\gamma)}(a_{i,j})$ is different from $0$.

The arc $\gamma$ belongs to $N_{E_u}$ if and only if its special point is sent inside $E_u$, and this happens precisely when 
\begin{equation}
h_{u,e}(Q_{0,m(\gamma)}(a_{i,j}):...:Q_{s,m(\gamma)}(a_{i,j}))=0.
\end{equation}

Hence, for any $M$, we define the system $S_2(M)$ 
of polynomial equations with coefficients in $\KK$ in the indeterminates $A_{i,j}$ 
\begin{equation}
\label{specialpointEu}
h_{u,e}(Q_{0,M}(A_{i,j}):...:Q_{s,M}(A_{i,j}))=0,
\end{equation}
for $e\in\{1,...,k_u\}$. Rephrasing the above discusion we have that $\gamma$ belongs to $N_{E_u}$ if and only if 
the coefficients of $\gamma$ satisfy the system $S_2(m(\gamma))$.

Let $U$ be a Zariski open subset of $E_u$. Let $C:=E_u\setminus U$. An argument similar to the one above shows that,
for any $M\in\NN$ there exists a system of finitely many polynomial equations $S_3(M)$ such that the lifting of 
an arc $\gamma\in\calX_\infty$ meets $C$ if and only if the coefficients of $\gamma$ satisfy the system $S_3(m(\gamma))$.

In order to characterise transverse arcs we recall the proof of Lemma~2.3~of~\cite{LR}.
Let $\nu$ be the divisorial valuation associated to $E_u$. Suppose, after a possible reordering, that $n_0:=\nu(g_0)=\min(\nu(g_0),...,\nu(g_s))$. There is a Zariski open subset $U$ of $E_u$ such that an arc
$\gamma$ has transverse lifting to $E_u$ through a point of $U$ if an only if $ord_t(g_0(\gamma))=n_0$. Moreover,
for any arc having lifting through $E_u$ we have $ord_t(g_0(\gamma))\geq n_0$. Rephrasing, if $\gamma$ is an arc with
lifting through $E_u$ then $\gamma$ lifts tranversely to $E_u$ if and only if the following polynomial inequality
is satisfied by the coefficients of $\gamma$:
\begin{equation}
\label{translr}
Q_{0,n_0}(A_{i,j})\neq 0.
\end{equation}

After this discussion the following proposition is obvious:

\begin{prop}
\label{nivelarcos}
For any essential divisor $E_u$ and a natural number $M$ we have:
\begin{enumerate}
\item the set $N(E_u,M)$ of arcs with lifting
through $E_u$ and first order with respect to $\pi$ equal to $M$ is the set of arcs in $\calX_\infty$ whose coefficients
satisfy the systems $S_1(M)$, $S_2(M)$ and do not satisfy the system $S_1(M+1)$. This is a finite ammount of polynomial equalities
and inequelities in the coefficients of the arc.
\item there is a Zariski open subset $U\subset E_u$ such that the set of arcs $\dot{N}(U,M)$ of arcs
with transverse lifting through $U$ and first order with respect to $\pi$ equal to $M$ is  a locally 
closed subset of $\calX_\infty$ defined by finitely many polynomial equalities and inequalities.
\end{enumerate}
\end{prop}

\subsubsection{Power series expansions of wedges realising adjacencies}
Given indeterminates $B_{i,j,k}$ with $1\leq i\leq N$ and $j,k\in\NN$ we the power series expansion
\[\Omega(t,s):=(\sum_{j,k\in\NN}B_{1,j,k}t^js^k,...,\sum_{j,k\in\NN}B_{i,j,k}t^js^k,...,\sum_{j,k\in\NN}B_{N,j,k}t^js^k).\]
is a $\KK[B_{i,j,k}]$-point in the space of wedges in $\KK^N$, which we call the universal wedge. 
A $K$-wedge $\alpha$ is obtained substituting the variables $B_{i,j,k}$ by elements $b_{i,j,k}\in K$. Its associated 
generic arc $\alpha_g(t)$ has coordinate power series
\begin{equation}
\label{coefsarcgen1}
\sum_{j\in\NN} a_{i,j}t^j
\end{equation}
with
\begin{equation}
\label{coefsarcgen2}
a_{i,j}=\sum_{k\in\NN}b_{i,j,k}s^k\in K[[s]].
\end{equation}

\begin{definition}
\label{2char}
Let $\alpha$ be a $K$-wedge. We define the {\em first order $m(\alpha)$ of $\alpha$ with respect to $\pi$} to be equal to the first order $m(\alpha_g)$ of its generic arc with respect to $\pi$.
\end{definition}

\begin{definition}
\label{3char}
Let $l$ be the minimal index such that $B_{l,m(\alpha_g)}(a_{i,j})\neq 0$. {\em The second order of $\alpha$ with
respect to $\pi$} is the order of $B_{l,m(\alpha_g)}$ in $s$, and is denoted by $R(\alpha)$.
\end{definition}

Since $B_{l,m(\alpha_g)}$ only depends on the coefficients $a_{i,j}$ when 
$j\leq m(\alpha_g)$, the coefficient of the term of order $R(\alpha)$ of $B_{0,m(\alpha_g)}$ is obtained by substitution
of the variables $B_{i,j,k}$ by the coefficients $b_{i,j,k}$ in a polynomial
\begin{equation}
\label{qwer}
P_\alpha(B_{i,j,k})
\end{equation}
in $\KK[B_{i,j,k}]$ for $j\leq m(\alpha_g)$ and $k\leq R(\alpha)$.

\begin{definition}
\label{4char}
{\em The characteristic polynomial $P_\alpha$ of $\alpha$ with respect
to $\pi$} is the polynomial in $\KK[B_{i,j,k}]$ defined in~\ref{qwer}.
\end{definition}

\begin{lema}
\label{wedgeporE_u}
Given a $K$-wedge $\alpha$ and a $K'$-wedge $\alpha'$ (with $K$ and $K'$ possibly different). Denote by $b'_{i,j,k}$ the coefficients of the coordinate 
series of $\alpha'$ and by $a'_{i,j}$ the coefficients of the coordinate series of $\alpha'_g$. If the systems of equations $S_1(m(\alpha))$ (see~(\ref{ordereqn})) and~$S_2(m(\alpha))$ (see~(\ref{specialpointEu})) are satisfied substituting
$A_{i,j}$ by $a'_{i,j}$, and we have
\begin{equation}
\label{ineqlocal}
P_{\alpha}(b'_{i,j,k})\neq 0
\end{equation}
then the generic arc $\alpha'_g$ belongs to $N_{E_u}$ and the first characteristic order of $\alpha'$ 
with respect to $\pi$ equal to $m(\alpha)$. 
\end{lema}
\begin{proof}
Since the system of equations $S_1(m(\alpha))$ is satisfied by the $a'_{i,j}$'s the first characteristic order of $\alpha'$ is at least $m(\alpha)$. Since
\begin{equation}
P_{\alpha}(b'_{i,j,k})\neq 0
\end{equation}
the system $S_1(m(\alpha))+1)$ is not satisfied by the the $a'_{i,j}$'s, 
and thus the first characteristic order of $\alpha'$ equals $m(\alpha)$. Since the system $S_2(m(\alpha))$ is satisfied 
by the $a'_{i,j}$'s, by Proposition~\ref{nivelarcos} the generic arc of $\alpha'$ belongs to $N_{E_u}$.
\end{proof}

\subsubsection{Wedges avoiding a closed subset}
\label{wedgesavoiding}

Let $\calZ$ be a proper closed subset of $\overline{N}_{E_u}$.
Let $\calI$ the ideal defining $\calZ$ in $\AAA_\KK^N$. Notice that the coordinate ring of the space of arcs in $\AAA_\KK^N$ is ring of polynomials in the indeterminates $A_{i,j}$ with coefficients in $\KK$.

Let $\alpha$ be a $K$-wedge with coordinate series
\[(\sum_{j,k\in\NN}b_{1,j,k}t^js^k,...,\sum_{j,k\in\NN}b_{i,j,k}t^js^k,...,\sum_{j,k\in\NN}b_{N,j,k}t^js^k).\]
Its generic arc $\alpha_g$ is of the form~(\ref{coefsarcgen1}), with the $a_{i,j}$ defined by Formula~(\ref{coefsarcgen2}).
If $\alpha_g$ is a $K((s))$-point of $N_{E_u}\setminus\calZ$ then there is a polynomial $H\in\calI$ such that 
$H(\alpha_g)=H(a_{i,j})$ is a non-zero power series in $K[[s]]$. Since $H$ depends on finitely many variables $A_{i,j}$ there is
a positive $l_1$ such that for any variable $A_{i,j}$ we have that $j\leq l_1$. 
Let $l_2$ be the order in $s$ of $H(a_{i,j})$. Clearly, there is a polynomial 
\begin{equation}
\label{polyavoiding}
\Phi_\alpha\in\KK[B_{i,j,k}]
\end{equation}
depending only on the 
variables $B_{i,j,k}$ for $i\leq N$, $j\leq l_1$ and $k\leq l_2$ such that the coefficient in $s^{l_2}$ is a non-zero element 
of $K$ obtained by substitution of $b_{i,j,k}$ into the polynomial $\Phi_\alpha$.

After this preparations we can prove Proposition~\ref{locallyclosed}.
\begin{proof}[Proof of Proposition~\ref{locallyclosed}]
Let $m(\alpha_s)$ be the first order of the special arc of the wedge $\alpha$ with respect to $\pi$. By Proposition~\ref{nivelarcos} the set $\dot{N}(U,m(\alpha_s))$ is a locally closed subset of $\calX_\infty$.
The mapping 
\[\calS:\calX^{Sing}_{\infty,\infty}\to\calX_\infty\]
assigning to any wedge its special arc is a morphism of infinite dimensional algebraic varieties. Therefore
$\Theta_1:=\calS^{-1}(\dot{N}(U,m(\alpha_s))))$ is a locally closed subset of $\calX^{Sing}_{\infty,\infty}$.

Let $m(\alpha)$ be the first order of the wedge $\alpha$ with respect to $\pi$. Since the generic arc $\alpha_g$ 
belongs to $N_{E_u}$, by Proposition~\ref{nivelarcos} the systems $S_1(m(\alpha))$ and $S_2(m(\alpha))$ are satisfied
by the coordinates $a_{i,j}=\sum_{k\in\NN}b_{i,j,k}s^k$ of the generic arc $\alpha_g$.

Let $P_\alpha\in\KK[b_{i,j,k}]$ be the characteristic polynomial of $\alpha$ with respect to $\pi$. Any wedge
$$\alpha'(t,s)=(\sum_{j,k\in\NN}b'_{1,j,k}t^js^k,...,\sum_{j,k\in\NN}b'_{N,j,k}t^js^k)$$ 
such that the coefficients of its generic arc satisfy system $S_1(m(\alpha))$ and such that the inequality
\begin{equation}
\label{inlr}
P_\alpha(b'_{1,j,k})\neq 0
\end{equation}
holds has first order with respect to $\pi$ equal to $m(\alpha')$.

We conclude that any wedge $\alpha'$ such that the systems $S_1(m(\alpha))$ and $S_2(m(\alpha))$ are satisfied
by the coordinates $a_{i,j}=\sum_{k\in\NN}b_{i,j,k}s^k$ of the generic arc $\alpha'_g$, and such that 
inequality~(\ref{inlr}) holds has a generic arc belonging to $N_{E_u}$. We denote by $\Theta_2$ the corresponding locally closed
subset of $\calX^{Sing}_{\infty,\infty}$.

The locally closed subset $\Lambda=\Theta_1\cap\Theta_2$ has the desired properties.
\end{proof}

\subsection{Specialisation of wedges}
Let $E_u$, $E_v$ be components of the exceptional divisor of the resolution $\pi$, being $E_v$ essential.
By Proposition~3.8~of~\cite{Re} the set $\overline{N}_{E_v}$ is generically stable. Let
$\calZ$ be any proper closed subset of $\overline{N}_{E_u}$. By Lemma~\ref{curveselectionlemma} and obtain a finite extension $K$ of $k_z$ and a 
$K$-wedge $\alpha$ whose special arc is the generic arc of $N_{E_v}$ and whose generic arc is belongs to 
$N_{E_u}\setminus\calZ$. We follow an specialisation method of~\cite{LR} to produce a wedge defined over the base field satisfying the same requirements. 

\begin{prop}
\label{especializacion}
Suppose that $\KK$ is uncountable. Suppose that
there is an adjacency from $E_u$ to $E_v$, that is we have the inclusion $N_{E_v}\subset\overline{N}_{E_u}$.
Let $\calZ$ be any proper closed subset of $\overline{N}_{E_u}$.
There exists a formal $\KK$-wedge realising an adjacency from $E_u$ to $E_v$ and avoiding $\calZ$.
\end{prop}
\begin{proof}
The proof follows the method of Section~2~of~\cite{LR}.
Let $z$ be the generic point of $\overline{N}_{E_v}$ and
$k_z$ its residue field. We apply Lemma~\ref{curveselectionlemma} and obtain a finite extension $K$ of $k_z$ and a 
$K$-wedge $\alpha$ whose special arc is the generic arc of $N_{E_v}$ and whose generic arc is belongs to 
$N_{E_u}\setminus\calZ$.

Since $\KK$ is uncountable, Proposition~2.10 of~\cite{I0} (which also holds for wedges) 
implies that the closed points of the space of wedges $\calX_{\infty,\infty}$ are $\KK$-valued points. In fact the proof of this Proposition gives that
the set of $\KK$-valued points is dense in $\calX_{\infty,\infty}$. In Proposition~\ref{locallyclosed} we have proved that there is a non-empty locally closed subset $\Lambda$ of wedges realising
an adjacency from $E_u$ to $E_v$ and avoiding $\calZ$. The density of $\KK$-valued points implies the existence of
a $\KK$-point in $\Lambda$, corresponds to a $\KK$-wedge with the desired properties.
\end{proof}

\section{Approximation of wedges}
\label{appwdg}
The aim of this section is to aproximate wedges realising adjacencies, avoiding closed sets and defined over the base 
field by wedges with the same properties, but defined by algebraic power series. The idea is to use the following 
Approximation Theorem, due to Becker, Denef, Lipshitz and van den Vries (Theorem~4.1~of~\cite{BDLD}), which later
was generalised by Popescu~\cite{Po}. I thank G. Rond and H. Hauser for pointing me out these approximation theorems. 

\begin{theo}
\label{Popescu}
Let $t,s$ and $y_1,...,y_r$ be variables. Consider polynomials $G_1,...,G_k\in\KK[t,s,y_1,...,y_r]$. Suppose 
that there exist formal power series $y_1(s),...,y_{r'}(s)$ and 
$y_{r'+1}(t,s),...,y_r(t,s)$ such that $G_i(t,s,y_1(s),...,y_r(t,s))=0$ for any $i\leq k$. For any $L>0$ there exist {\em algebraic} power series $y'_1(s),...,y'_{r'}(s)$ and $y_{r'+1}(t,s),...,y_r(t,s)$ such that
$G_i(t,s,y'_1(s),...,y'_r(t,s))=0$ for any $i\leq k$ and $y'_j$ coincides with $y_j$ up to order $L$. 
\end{theo}

\begin{prop}
\label{approximation}
Given any $\KK$-wedge $\alpha$ realising an adjacency from $E_u$ to $E_v$, and any natural number $L$, there exists an
algebraic wedge $\beta$ realising an adjacency from $E_u$ to $E_v$ such that their power series expansion coincide up to order $L$. Moreover if $\calZ$ is a proper closed subset of $\overline{N}_{E_u}$ and $\alpha$ avoids it then
$\beta$ can be chosen avoiding $\calZ$.
\end{prop}
\begin{proof}
If the generic arc $\alpha_g$ does not belong to $\calZ$ then set $l_1$, $l_2$ and $\Phi_\alpha$ be the natural numbers
and the polynomial introduced in~\ref{wedgesavoiding}. Otherwise set $l_1=l_2=0$.
Recall from Subsection~\ref{subsubarcos} that $n_0:=min(\nu(g_0),...,\nu(g_s))$.

We define
\[m:=\max\{m(\alpha_s),m(\alpha_g)\}\]
to be the maximum of the orders of the special and generic arcs with respect to $\pi$ and assume 
\[L>\max\{m,R(\alpha),l_1,l_2,n_0\}.\]

We write the wedge $\alpha$ according with the expression of its generic arc as
\[\alpha(t):=(\sum_{j\in\NN}a_{1,j}(s)t^j,...,\sum_{j\in\NN}a_{N,j}(s)t^j),\]
with $a_{i,j}(s)=\sum_{k\in\NN}b_{i,j,k}s^k$.
We re-group this expansion in the following way:
\begin{equation}
\label{formaadecuada}
\alpha(t,s):=(\sum_{j=1}^{m}a_{1,j}(s)t^j+c_1(t,s)t^{m+1},...,\sum_{j=1}^{m}a_{N,j}(s)t^j+c_N(t,s)t^{m+1}),
\end{equation}
with $c_i(s,t)\in\KK[[t,s]]$.

Consider variables $A_{i,j}$ and $C_i$ for $1\leq i\leq N$ and $1\leq j\leq m$. Given any generator $F_l$ of the ideal $I$ of $X$ we consider the polynomial
\[G_l(t,A_{i,j},C_k):=F_l((\sum_{j=1}^{m}A_{1,j}t^j+C_1t^{m+1},...,\sum_{j=1}^{m}A_{N,j}t^j+C_Nt^{m+1})\]
in $\KK[t,A_{i,j},C_k]$.

The fact that the image of the wedge $\alpha$ lies in $X$ is expressed by the system of equations
\begin{equation}
\label{inX}
G_l(t,a_{i,j}(s),c_k(t,s))=0,
\end{equation}
for $1\leq l\leq r$. Moreover, since the wedge $\alpha$ realises an adjacency from $E_u$ to $E_v$, the general arc $\alpha_g$ belongs to 
$N_{E_u}$, and hence the power series $a_{i,j}(s)$ satisfy the systems of
polynomial equations $S_1(m(\alpha))$ (see~(\ref{ordereqn})) and $S_2(m(\alpha))$ (see~(\ref{specialpointEu})).

By Theorem~\ref{Popescu} there exist algebraic power series $a'_{i,j}(s)$ and $c'_i(t,s)$ coinciding
respectively with $a_{i,j}(s)$ and $c_i(t,s)$ up to order $L$, for $1\leq i\leq N$ and $1\leq j\leq m$, such that they satisfy the systems of 
equations~(\ref{inX}), $S_1(m(\alpha))$ and $S_2(m(\alpha))$.

We claim that
\begin{equation}
\label{wedgealg}
\alpha'(t,s):=(\sum_{j=1}^{m}a'_{1,j}(s)t^j+c'_1(t,s)t^{m+1},...,\sum_{j=1}^{m}a'_{N,j}(s)t^j+c'_N(t,s)t^{m+1})
\end{equation}
is the required algebraic wedge.

We expand $\alpha'$ as
\[\alpha'(t,s)=(\sum_{i,j}b'_{i,j,1}t^is^j,...,\sum_{i,j}b'_{i,j,N}t^is^j).\]

The wedge $\alpha'$ is algebraic since it is a polynomial combination of algebraic series and its image lies in $X$ since it satisfy the system of 
equations~(\ref{inX}). The systems of equations~$S_1(m(\alpha))$ and $S_2(m(\alpha))$ 
are satisfied for the coefficients of $\alpha'_g$ by construction.

As $\alpha$ and $\alpha'$ coincide up to order $L$ we have $b_{i,j,k}=b'_{i,j,k}$ for $i\leq N$ and 
$j\leq M$. Since the characteristic polynomial of $\alpha$ with respect to $\pi$ only involves variables $b_{i,j,k}$
with $i\leq N$, $j\leq m(\alpha)$, 
$k\leq R(\alpha)$, and $M>\max{m(\alpha),R(\alpha)}$ we have
\[P_\alpha(b'_{i,j,k})=P_\alpha(b_{i,j,k})\neq 0.\]
Thus, since the system $S_1(m(\alpha))$ is satisfied by the coefficients of $\alpha'_g$, the first order of $\alpha'$
with respect to $\pi$ is equal to $m(\alpha)$. Hence, since the system $S_2(m(\alpha))$ is satisfied the generic arc of
$\alpha'$ belongs to $N_{E_u}$ (see Proposition~\ref{nivelarcos}).

Let $H$, $l_1$, $l_2$ and $\Phi$ as in~\ref{wedgesavoiding}. Let the generic arc $\alpha'_g$ is a $K((s))$-arc, and
hence the evaluation $H(\alpha'_g)$ is a power series belonging to $K((s))$. Its coefficient in $l_2$ is given by the 
substitution in $\Phi_\alpha$ of the variables $B_{i,j,k}$ by the coefficients $b'_{i,j,k}$. Since $\Phi_\alpha$ only depends on the
variables $B_{i,j,k}$ for $i\leq N$ and $j,k\leq L$ and we have the equality $b_{i,j,k}=b'_{i,j,k}$ in this range of 
indices we find that the coefficient is non-zero as with the case of $H(\alpha_g))$. This shows that $\alpha'_g$ does not
belong to $\calZ$.

We know that $\alpha'_s$ belongs to $\dot{N}_v$ because it coincides with $\alpha_s$ up to order $M>n_0$, and the transverseity is controlled by the non vanishing of the substitution of the coefficients of the special arc of 
$\alpha'$ in $Q_{0,n_0}(A_{i,j})$ (see Subsection~\ref{subsubarcos}).
\end{proof}

\section{Algebraic wedges and finite morphisms}
\label{awfm}

We need to refine at this point the definition of wedges realising adjacencies:

\begin{definition}
\label{wedgeadj}
Let $\gamma$ be a $\KK$-arc in $\dot{N}_{E_v}$ and $\calZ$ a proper closed subset of $\overline{N}_{E_u}$. 
A $\KK$-wedge $\alpha$ realises an adjacency from $E_u$ to $\gamma$ (avoiding $\calZ$) if its generic arc belongs to $N_{E_u}$ (and not to $\calZ$) and its special arc coincides with $\gamma$.
\end{definition}

Along this Section 
$$\pi:\tilde{X}\to X$$ 
denotes the minimal good resolution of the singularity at $O$. 

\begin{definition}
\label{branchadj}
Let $\gamma$ be an algebraic $\KK$-arc in $\dot{N}_{E_v}$.
A morphism (in the category of algebraic varieties)
\[\psi:(Z,Q)\to (X,O)\]
of germs of normal surface singularities realises an adjacency from $E_u$ to $\gamma$ if there exist a resolution of singularities
\[\rho:\tilde{Z}\to (Z,Q),\]
of the singularity of $Z$ at $Q$, with exceptional divisor $F:=\rho^{-1}(0)$ and an algebraic $\KK$-arc 
\[\varphi:Spec(\KK[[t]])\to\tilde{Z}\]
satisfying the equality $\psi\comp\rho\comp\varphi=\gamma$ and the following additional properties:
\begin{enumerate}
\item the morphism $\rho$ resolves the indeterminacy of the rational map $\pi^{-1}\comp\psi$; that is, the rational map $\tilde{\psi}:=\pi^{-1}\comp\psi\comp\rho$ is a
morphism.
\item There is an irreducible component $F_0$ such that:
\begin{enumerate}[(a)]
\item The image $\tilde{\psi}(F_0)$ is contained in $E_u$.
\item The resolution $\rho$ can be factored into $\rho_2\comp\rho_1$, where
\[\rho_1:\tilde{Z}\to W\]
is a the sequence of contractions of $(-1)$-curves which does not collapse $F_0$, and 
\[\rho_2:W\to (Z,Q)\]
is another resolution of singularities. The arc $\rho_1\comp\varphi$ sends the special point into a point $Q'$ of $F_0$ in which $F_0$
has a smooth branch $L$. Moreover $\rho_1\comp\varphi$ is transverse to $L$ at $Q'$.
\end{enumerate}
\end{enumerate}

Let $\calZ$ be a proper closed subset of $\overline{N}_{E_u}$.
We say that the morphism $\phi$ {\em realises an adjacency from $E_u$ to $\gamma$ avoiding $\calZ$} if there exist a 
resolution of singularities $\rho$, an algebraic $\KK$-arc $\varphi$, a factorisation $\rho=\rho_2\comp\rho_1$, an irreducible component $F_0$ of the exceptional divisor of $\rho$, a smooth branch germ $(L,Q')$ into $\rho_1(F_0)$ with
the properties predicted above, and such that there exists a morphism 
\[\beta:Spec(\KK[[t,s]])\to (W,Q')\] 
such that:
\begin{enumerate}[(i)]
\item The $\KK$-arc $\beta(t,0)$ coincides with $\rho_1\comp\varphi(t)$.
\item The restriction $\beta(0,s)$ is a formal parametrisation of the germ $(L,Q')$.
\item The composition $\psi\comp\rho_2\comp\beta$ is a $\KK$-wedge whose generic arc does not belong to $\calZ$.
\end{enumerate}
\end{definition}

\begin{definition}
\label{branchadj2}
A morphism (in the category of algebraic varieties)
\[\psi:(Z,Q)\to (X,O)\]
of normal surface singularities realises an adjacency from $E_u$ to $E_v$ (avoiding an proper closed subset $\calZ$
of $\overline{N}_{E_u}$) if there exists an algebraic $\KK$-arc $\gamma$ in $\dot{N}_{E_v}$ for which
the morphism realises an adjacency from $E_u$ to $\gamma$ (avoiding $\calZ$). 
\end{definition}

\begin{definition}
\label{branchadjC}
Suppose $\KK=\CC$.
\begin{enumerate}
\item Let $\gamma$ be a convergent $\CC$ arc in $\dot{N}_{E_v}$. An analytic mapping germ
\[\psi:(Z,Q)\to (X,O)\]
of normal surface singularities realises an adjacency from $E_u$ to $\gamma$ (avoiding an proper closed subset $\calZ$
of $\overline{N}_{E_u}$) if it satisfies the conditions of Definition~\ref{branchadj} in the complex analytic category.
\item An analytic mapping germ
\[\psi:(Z,Q)\to (X,O)\]
of normal surface singularities realises an adjacency from $E_u$ to $E_v$ (avoiding an proper closed subset $\calZ$
of $\overline{N}_{E_u}$) if there exists a convergent $\CC$-arc $\gamma$ in $\dot{N}_{E_v}$ for which
the mapping realises an adjacency from $E_u$ to $\gamma$ (avoiding $\calZ$).
\end{enumerate}
\end{definition}

\begin{prop}
\label{wedcov}
Let $\gamma$ be an algebraic $\KK$-arc in $\dot{N}_{E_v}$ and $\calZ$ a proper closed subset of $\overline{N}_{E_u}$.
The following holds:
\begin{enumerate}
\item If there exists a morphism realising an adjacency from $E_u$ to $\gamma$ (avoiding $\calZ$) then there exists a
$\KK$-wedge realising an adjacency from $E_u$ to $\gamma$ (avoiding $\calZ$).
\item If there exists an algebraic $\KK$-wedge realising an adjacency from $E_u$ to $\gamma$ (avoiding $\calZ$) 
then there exists a {\em finite} morphism realising an adjacency from $E_u$ to $\gamma$ (avoiding $\calZ$).
\end{enumerate}
\end{prop}
\begin{proof}

Let
\[\psi:(Z,Q)\to (X,O)\]
be a morphism realising an adjacency from $E_u$ to $\gamma$. As $W$ is smooth and $\rho_1\comp\varphi$ is transverse to $L$ at $Q'$
(Definition~\ref{branchadj},~(2),~(b)), there exists a scheme isomorphism
\[\beta:Spec(\KK[[t,s]])\to\widehat{(W,Q')}\]
(where $\widehat{(W,Q')}$ is the formal neighborhood of $Q'$ at $W$), 
such that the $\KK$-arc $\beta(t,0)$ coincides with $\rho_1\comp\varphi(t)$,
and $\beta(0,s)$ is a formal parametrisation of $(L,Q')$. The mapping
\[\alpha:=\psi\comp\rho_2\comp\beta:Spec(\KK[[t,s]])\to (X,O)\]
is a wedge since $\rho_2(0,s)=Q$ for any $s$. Its generic arc belongs to $N_{E_u}$ since $\tilde{\psi}(F_0)$ is contained in $E_u$, and its special arc
is $\psi\comp\rho_2\comp\beta(t,0)$. We have 
\[\psi\comp\rho_2\comp\beta(t,0)=\psi\comp\rho_2\comp\rho_1\comp\varphi(t)=\psi\comp\rho\comp\varphi=\gamma.\]
Thus $\alpha$ is a wedge realising an adjacency from $E_u$ to $\gamma$. 

If the morphism $\psi$ also avoids $\calZ$ we take as $\beta$ the morphism predicted in Definition~\ref{branchadj}.
This proves (1).

Let
\[\alpha:Spec(\KK[[t,s]])\to (X,O)\subset\KK^N\]
be an algebraic $\KK$-wedge realising an adjacency from $E_u$ to $\gamma$.

As $\alpha$ is algebraic, for any $i$, there exists a polynomial $G_i\in\KK[t,s,x_i]$ such that 
\[G_i(t,s,x_i(\alpha)(t,s))=0.\]
The polynomials $G_1,...,G_N$ define an algebraic set $Y_1$ in the affine space $\KK^2\times\KK^N$. Let 
\[\alpha':Spec(\KK[[t,s]]\to \KK^2\times\KK^N\]
be the graph of $\alpha$ (that is $\alpha'(t,s):=(t,s,\alpha(t,s)$). Let $Y_2$ be the ($2$-dimensional) irreducible component of $Y_1$ containing the image of the graph
of the mapping $\alpha$. Denote by $n:Y_3\to Y_2$ the normalisation. The mapping $\alpha'$ clearly admits a lifting 
\[\beta:Spec\KK[[t,s]])\to Y_3,\]
which is a formal isomorphism at the origin $O'$ of $Spec(\KK[[t,s]])$.
Thus, we may see $Spec(\KK[[t,s]])$ as the formal neighborhood of $Y_3$ at the smooth point $\beta(O')$.

Denote by $pr_2$ the projection of $\KK^2\times\KK^N$ to the second factor.
We consider normal projective completions $\overline{X}$ and $\overline{Y}_3$ of $X$ and $Y_3$ respectively such that the mapping 
\[pr_2|_{Y_2}\comp n:Y_3\to X\]
extends to a projective morphism
\[\phi:\overline{Y}_3\to \overline{X}.\]
By construction we have the equality
\begin{equation}
\label{comm1}
\phi\comp\beta=\alpha.
\end{equation}

Using Stein factorisation (Corollary~III.11.5~of~\cite{Ha}) we can factor $\phi$ into $\psi\comp\sigma$, where 
\[\sigma:\overline{Y}_3\to Z\]
is a proper birational morphism, the variety $Z$ is a normal projective surface, and 
\[\psi:Z\to\overline{X}\]
is a finite morphism.

Denote by $Q$ the point $\sigma(\beta(O'))$. The restriction 
\begin{equation}
\label{branchcov}
\psi:(Z,Q)\to (X,O)
\end{equation}
is a finite morphism of normal surface singularities. Let 
\[\rho:\tilde{Z}\to (Z,Q)\]
be the minimal resolution of singularities such that
\begin{enumerate}[(i)]
\item it resolves the indeterminacy of the rational map $\pi^{-1}\comp\psi$ (that is, the rational map $\tilde{\psi}:=\pi^{-1}\comp\psi\comp\rho$ is a
 morphism),
\item it dominates $\sigma$.
\end{enumerate}
Let 
\[\sigma':W\to \overline{Y}_3\]
be the minimal resolution of singularities of $\overline{Y}_3$. The composition $\rho_2:=\sigma\comp\sigma'$ is a resolution of singularities of $(Z,Q)$ and
the resolution $\rho$ factors into $\rho_2\comp\rho_1$, where $\rho_1$ is a the sequence of contractions of $(-1)$-curves. As $\overline{Y}_3$ is smooth
at the point $\beta(O')$ there is a unique point $Q'\in W$ such that $\sigma'(Q')=\beta(O')$, and the mapping $\beta$ admits an isomorphic lifting 
\begin{equation}
\label{lifting'}
\beta':Spec(\KK[[t,s]])\to (W,Q').
\end{equation}

Since we have the equalities
\begin{equation}
\label{igualdades}
\psi\comp\rho_2\comp\beta'=\psi\comp\sigma\comp\sigma'\comp\beta'=\psi\comp\sigma\comp\beta=\phi\comp\beta=\alpha
\end{equation}
and $\psi$ is finite we have that $\rho_2\comp\beta'$ transforms the $s$-axis into $Q$.
Thus the image $L$ of the $s$-axis by $\beta'$ belongs to the exceptional divisor of $\rho_2$.
Let $F'_0$ be the irreducible component of the exceptional divisor of $\rho_2$ containing the smooth branch $L$. Denote by $F_0$ the strict transform of
$F'_0$ by $\rho_1$. By construction $\rho_1$ does not collapse $F_0$. As $\alpha$ realises an adjacency from
$E_u$ to $\gamma$, it sends the special point of its general arc into $E_u$. This implies that the divisor $F_0$ is transformed into $E_u$ by $\tilde{\psi}$.
On the other hand the equalities~(\ref{igualdades}) imply that the arc $\varphi$ predicted in Definition~\ref{branchadj},~(2) is the lifting of the 
special arc of $\beta'$ to $\tilde{Z}$.

If the wedge $\alpha$ avoids $\calZ$ then the lifting $\beta'$ defined in~(\ref{lifting'}) is the additional morphism
predicted in Definition~\ref{branchadj}.
\end{proof}

\section{Moving wedges}
\label{mw}

In this section we assume that the base field $\KK$ is the field of complex numbers $\CC$ and we will work in the complex analytic category. Thus $(X,O)$ denotes a complex analytic normal surface singularity germ. Since $(X,O)$ has 
an isolated singularity at the origin, a result of Artin~\cite{Ar} shows that it is formally isomorphic with a germ of algebraic surface at a point. Then, by a
Theorem of Hironaka and Rossi~\cite{Hi},
we know that the formal isomorphism comes in fact from a local analytic one. Thus, when we need it we will freely assume that $(X,O)$ is an algebraic 
surface germ. It is worth to remind, for intuition, that in the category of normal complex analytic spaces, finite 
analytic morphisms correspond topologically to branched coverings.

We develope a technique that allows, given a $\CC$-wedge realising an adjacency from 
$E_u$ to a particular $\CC$-arc of $\dot{N}_{E_v}$, to construct wedges realising adjacencies from $E_u$ to {\em any} 
convergent $\CC$-arc in $\dot{N}_{E_v}$. The technique does not allow to ensure that if the original wedge avoids a
proper closed subset $\calZ$ of $\overline{N}_{E_u}$ then the newly constructed wedge also enjoys this property.
However, if the subset $\calZ$ equals the set of non-transverse arcs $\Delta_{E_u}$ we will prove that the property of
avoiding it is preserved.

\begin{definition}
Given any irreducible component $E_u$ of the exceptional divisor we define the proper closed subset 
$\Delta_{E_u}\subset N_{E_u}$ consisting of the union of the set of arcs whose lifting to the resolution sends
the special point to a singular point of $E$ with the set of arcs whose lifting is non-transverse to $E_u$. 
\end{definition}

We need to remind some common notations on dual graphs and the plumbing construction. Given a good resolution 
\[\pi:\tilde{X}\to (X,O)\]
of a normal surface singularity, the dual graph $\calG$ of the resolution is constructed as follows: there is a vertex for each irreducible component $E_i$ 
of $E$, with a genus weight (which is equal to the genus of $E_i$), and a self-intersection weight (which is equal to the self intersection of $E_i$in $\tilde{X}$). Two 
vertices are joined by one edge for each intersection point of their corresponding exceptional divisors. In this article a {\em weighted graph} is a doubly weighted graph
as above. Let $r$ be the number of vertices of a weighted graph $\calG$. The incidence matrix of $\calG$ is a symmetric square $r\times r$ matrix with the following
entries: enumerate the vertices of $\calG$, the $(i,i)$-entry is the self-intersection weight of the $i$-th vertex; if $i\neq j$ the $(i,j)$ entry is the number of edges between the $i$-th and $j$-th vertices.
It is well known~\cite{Gr} that a weighted graph is associated to a good resolution of a normal surface singularity if and only
if its incidence matrix is negative-definite. In that case we say that the graph is a negative definite weighted graph. Another well known fact is that an adequate
neighborhood of $E$ in $\tilde{X}$ is is diffeomorphic to the plumbing $4$-manifold associated to the dual graph of the resolution~\cite{Ne}.

The main technical difficulty is solved in:

\begin{prop}
\label{movinbc}
If there exists a convergent $\CC$-arc $\gamma\in\dot{N}_{E_v}$ and a finite analytic mapping realising and adjacency
from $E_u$ to $\gamma$ (avoiding $\Delta_{E_u}$), then, for any other convergent $\CC$-arc $\gamma'\in\dot{N}_{E_v}$
there exists a finite analytic mapping realising an adjacency from $E_u$ to $\gamma'$ (avoiding $\Delta_{E_u}$).
\end{prop}

The result which allows to move wedges is:

\begin{prop}
\label{movinw}
If there exists a $\CC$-arc $\gamma\in\dot{N}_{E_v}$ and a $\CC$-wedge realising and adjacency from $E_u$ to $\gamma$
(avoiding $\Delta_{E_u}$), then, for any convergent $\CC$-arc $\gamma'\in\dot{N}_{E_v}$ there exists a $\CC$-wedge
realising an adjacency from $E_u$ to $\gamma'$ (avoiding $\Delta_{E_u}$).
\end{prop}
\begin{proof}
Let us assume Proposition~\ref{movinbc}. By Proposition~\ref{approximation} there exists an algebraic $\CC$-wedge
$\alpha$ realising an adjacency from 
$E_u$ to $E_v$ (avoiding $\Delta_{E_u}$) . If $\gamma''$ denotes the special arc of $\alpha$ then $\alpha$ realises an adjacency from $E_u$ to $\gamma''$ (avoiding $\Delta_{E_u}$). By 
Proposition~\ref{wedcov}~(2) there exists a finite morphism realising an adjacency from $E_u$ to $\gamma''$ (avoiding $\Delta_{E_u}$). By Proposition~\ref{movinbc}
for any convergent $\gamma'\in\dot{N}_{E_v}$ there exists a finite morphism realising an adjacency from $E_u$ to $\gamma'$ (avoiding $\Delta_{E_u}$) . Finally, by Proposition~\ref{wedcov}~(1) for any 
convergent $\gamma'\in\dot{N}_{E_v}$ there exists a $\CC$-wedge realising an adjacency from $E_u$ to $\gamma'$ (avoiding $\Delta_{E_u}$).
\end{proof}

\begin{proof}[Proof of Proposition~\ref{movinbc}]
The proof is somewhat involved but the idea is simple: let
\[\psi:(Z,Q)\to (X,O)\]
be a finite analytic morphism of normal surface singularities realising an adjacency from $E_u$ to $\gamma$. Since 
$\gamma$ and $\gamma'$ are arcs with transverse lifting through the same essential component there exists a 
self-homeomorphism $\Phi$ from of $(X,O)$ transforming the arc $\gamma$ into the arc $\gamma'$. The composition of 
$\Phi\comp\psi$ is a topological branched cover. Since the conditions of the definition of
finite analytic morphisms realising an adjacency are of a topological nature, if one could put another analytic
structure on $(Z,Q)$ so that $\Phi\comp\psi$ is holomorphic, then $\Phi\comp\psi$ would be a finite analytic mapping
realising an adjacency from $E_u$ to $\gamma'$. 

The proof has three parts, the first is a construction of the self-homeomorphism 
which allows to change the complex structure in $(Z,Q)$ so that the composition becomes holomorphic, 
the second consists in constructing the new complex structure, and the third consists in checking that our 
construction gives, in fact, a finite analytic mapping realising an adjacency from $E_u$ to $\gamma'$.\\

\noindent{\bf PART I}.

We will construct the homeomorphism at the level of the resolution $\tilde{X}$.

Let
\[\psi:(Z,Q)\to (X,O)\]
be a finite analytic morphism of normal surface singularities realising an adjacency from $E_u$ to $\gamma$. Let 
\[\rho:\tilde{Z}\to (Z,Q)\]
be the resolution of singularities predicted in Definition~\ref{branchadj}, and
$\tilde{\psi}:\tilde{Z}\to\tilde{X}$ the lifting of $\psi$. By Stein factorisation $\tilde{\psi}$ factors into $\tilde{\psi}'\comp\rho'$, where
\[\tilde{\psi}':Z'\to\tilde{X}\]
is a finite analytic mapping, with $Z'$ a normal surface, and 
\[\rho':\tilde{Z}\to Z'\]
a resolution of singularities. Moreover the resolution $\rho$ factors into $\rho''\comp\rho'$, where
\[\rho'':Z'\to (Z,Q)\]
is a bimeromorphic morphism.

Let $\Delta$ be the branching locus of $\tilde{\psi}'$, and $\Delta'$ be the union of the irreducible components of $\Delta$ not contained in the exceptional
divisor $E$. Let $\tilde{\gamma}$ and $\tilde{\gamma}'$ be liftings of $\gamma$ and $\gamma'$ to $\tilde{X}$.
Since $\tilde{X}$ is diffeomorphic to the plumbing $4$-manifolds of the dual graph of the resolution, possibly having to shrink it to a smaller tubular neighborhood 
of the exceptional divisor of $\pi$ we may assume:
\begin{enumerate}
\item There are neighborhoods $U_1,...,U_m$ of each singular point of the exceptional divisor in $\tilde{X}$ which
contain all the components of $\Delta'$ which meet the corresponding singular point.
\item The difference $\tilde{X}\setminus \cup_{j=1}^mU_j$ splits in one connected component $\tilde{X}_v$ for each
irreducible component $E_v$ of $E$. 
In addition, for any $v$ we have that $\hat{E}_v:=E_v\setminus\cup_{j=1}^mU_j$ is the
deletion of several discs in $E_v$, and we have a smooth
product structure $\tilde{X}_v=\hat{E}_v\times D$, where $D$ is a disk.
\item Consider the set $\{p_1,...,p_s\}$ of meeting points of $\Delta'$ with $E$ which are not singularities of $E$. 
Given any point $p_i$ contained in a divisor $E_u$, there exists a disk $D_i$ around 
$p_i$ in $\hat{E_u}$ such that all the components of $\Delta'$ meeting $p_i$ are contained in $D_i\times D$, and any 
component of $\Delta'$ meeting $D_i\times D$ meet $E$ at $p_i$. In addition the closure of the discs $D_1$,...,$D_s$
are two-by-two disjoint and do not meet the boundary of $\hat{E}_v$.
\item The images of $\tilde{\gamma}$ and $\tilde{\gamma}'$ are fibre disks in $\tilde{X}_v$ by the projection to 
$\hat{E}_v$. Either there is a disk
$D_{\tilde{\gamma}}$ in $\hat{E}_v$ around $\tilde{\gamma}(0)$ with closure disjoint to the closure of any $D_i$ 
and to the boundary of $\hat{E}_v$, or $\tilde{\gamma}$ is the fibre of a point in $\{p_1,...,p_s\}$ by the projection 
to $\hat{E}_v$. In the later case we define the disk $D_{\tilde{\gamma}}$ to be the equal to $D_j$, for $\tilde{\gamma}(0)=p_j$.
The same holds for $\tilde{\gamma}'$. The closures 
of the two disks $D_{\tilde{\gamma}}$ and $D_{\tilde{\gamma}'}$ are disjoint and disjoint to the boundary of 
$\hat{E}_v$.
\end{enumerate}

Denote by $\calD$ an open region of $\hat{E}_v$ diffeomorphic to a disk such that the closure of the disks $D_{\tilde{\gamma}}$ and $D_{\tilde{\gamma}'}$
are contained in $\calD$ and the closure of $\calD$ is disjoint to the boundary of $\hat{E}_v$ and to the closures of
the disks $D_i$ different from $D_{\tilde{\gamma}}$ and $D_{\tilde{\gamma}'}$.

The construction of the homeomorphism is achieved in the following lemma.
\begin{lema}
\label{difeolr}
There exists a diffeomorphism 
$$\eta:\tilde{X}\to\tilde{X}$$
such that
\begin{enumerate}[(i)]
\item it leaves $E$ invariant,
\item the restriction of $\eta$ to $\tilde{X}\setminus(\calD\times D)$ is the identity,
\item the restrictions
\[\eta_1:D_{\tilde{\gamma}}\times D\to D_{\tilde{\gamma}'}\times D\]
\[\eta_2:D_{\tilde{\gamma}'}\times D\to D_{\tilde{\gamma}}\times D\]
are biholomorphisms such that we have the equalities $\eta\comp\tilde{\gamma}=\tilde{\gamma}'$ and 
$\eta\comp\tilde{\gamma}'=\tilde{\gamma}$.
\end{enumerate}
\end{lema}
\begin{proof}
The diffeomorphism can be constructed as follows: define first a biholomorphism 
\[\xi:\overline{D_{\tilde{\gamma}}}\to \overline{D_{\tilde{\gamma}'}}\]
taking $\tilde{\gamma}(0)$ to $\tilde{\gamma}'(0)$. Now define
\[\eta':D_{\tilde{\gamma}}\times D\to D_{\tilde{\gamma}'}\times D\]
by $\eta'(x,y):=(\xi(x),y)$. The composition $\eta'\comp\tilde{\gamma}$ is a re-parametrisation of $\tilde{\gamma}'$: in
fact there exists a biholomorphic mapping $\zeta$ such that we have 
$(id,\zeta)\comp\eta'\comp\tilde{\gamma}=\tilde{\gamma}'$.
Let $\overline{\xi}:\calD\to\calD$ be any diffeomorphism whose restrictions to $\overline{D_{\tilde{\gamma}}}$ and 
$\overline{D_{\tilde{\gamma}'}}$ coincides with $\xi$ and $\xi^{-1}$ respectively, and which is the identity in a 
neighborhood of the boundary of $\calD$. Consider a family of diffeomorphisms $\Psi_x$ of $D$ 
parametrised over $\calD$ such that
\begin{itemize}
\item if $x\in D_{\tilde{\gamma}}$ then $\Psi_x$ coincides with $\zeta$,
\item if $x\in D_{\tilde{\gamma}'}$ then $\Psi_x$ coincides with $\zeta^{-1}$, 
\item if $x$ belongs to a certain neighborhood of the boundary of $\calD$ then $\Psi_x$ is the identity.
\end{itemize}
We define
\[\eta(x,y):=(\overline{\xi}(x),\Psi_x(y)),\]
if $(x,y)\in\calD\times D$, and $\eta(p)=p$ if $p$ is outside $\calD\times D$.
\end{proof}

The Lemma defines the homeomorphism at the level of the resolution. Since it leaves $E$ invariant it may be push down
to a self-homeomorphism of 
$$\nu:(X,O)\to (X,O).$$

\noindent{\bf PART II.}

In this part we define a new holomorphic structure in $(Z,Q)$ which makes the composition
$$\nu\comp\pi:(Z,Q)\to (X,O)$$
an analytic morphism. Since in the process we need to deal often with the same topological space with different 
complex structures we will denote an analytic set by a pair
$(Y,\calB)$, where $Y$ denotes the underlying topological space and $\calB$ the complex structure.

The original complex structure of the surface germ $(Z,Q)$ is denoted by $\calD_1$. 
\begin{lema}
\label{newcompstr}
There exists a unique complex structure $\calD_2$ in the germ of topological space $(Z,Q)$ such that
$$\nu\comp\psi:(Z,Q,\calD_2)\to (X,O)$$
is a finite analytic morphism
\end{lema}
\begin{proof} The new complex structure is constructed first at the level of the resolution of $(Z,Q)$ and pushed down later.
The finiteness of $\nu\comp\psi$ is a topological property and holds because $\psi$ is finite.

\textsc{Step 1}. The mapping
\begin{equation}
\label{branchcomp}
\eta\comp\tilde{\psi}':Z'\to\tilde{X}
\end{equation}
is a topological branched covering, which is a local diffeomorphism outside $\eta(\Delta)$, but it is not holomorphic.
We are going to
change the analytic structure in $Z'$ so the mapping $\eta\comp\tilde{\psi}'$ becomes holomorphic.
We endow $Z'\setminus (\tilde{\psi}')^{-1}(\Delta)$ with the unique
complex structure making the restriction $\eta\comp\tilde{\psi}'|_{Z'\setminus (\tilde{\psi}')^{-1}(\Delta)}$
holomorphic. This
restriction becomes automatically a finite and \'etale analytic morphism with the new complex structure.
Notice that by construction of $\eta$ the branch locus $\eta(\Delta)$ of $\eta\comp\tilde{\psi}'$ is an analytic subset.
Thus, by a Theorem of Grauert and Remmert~\cite{GrR}, there exists a unique normal
complex analytic set $Z''$ containing $Z'\setminus (\tilde{\psi}')^{-1}(\Delta)$ for which 
$\eta\comp\tilde{\psi}'|_{Z'\setminus (\tilde{\psi}')^{-1}(\Delta)}$ extends to a finite analytic mapping
\begin{equation}
\label{newbranch}
\tilde{\psi}'':Z''\to\tilde{X}
\end{equation}
Since a branched covering is determined as a topological space by its restriction over the complement of the branching
locus we have that $Z'$ and $Z''$ coincide as topological spaces and that $\tilde{\psi}''$ equals
$\eta\comp\tilde{\psi}'$. Let $\calA_1$ denote the original complex structure of $Z'$ and $\calA_2$ denote the complex structure
of $Z''$. We will denote $Z'$ by $(Z',\calA_1)$ and $Z''$ by $(Z',\calA_2)$.

\textsc{Step 2}. The singularities of $(Z',\calA_1)$ are over the singular points of the discriminant. 
Since these points are mapped by $\tilde{\psi}'$ to points where the mapping $\eta$ is biholomorphic, the singularities
of $(Z',\calA_2)$ are analytically equivalent to those of $(Z',\calA_1)$. Let
\[\rho_{min}:(Z_{min},\calB_1)\to (Z',\calA_1)\]
be the minimal resolution of singularities (we denote by $\calB_1$ the complex structure of $Z_{min}$).
If we change the analytic structure $\calA_1$ by $\calA_2$ then the mapping
$\rho_{min}$ is holomorphic over neighborhoods of the singularities of $(Z',\calA_2)$, and is a non-necessarily holomorphic diffeomorphism outside these neighborhoods. Thus there is a unique complex structure $\calB_2$ in 
$Z_{min}$ making 
\[\rho_{min}:(Z_{min},\calB_2)\to (Z',\calA_2)\] 
holomorphic.

\textsc{Step 3}.
Let $\calC_1$ denote the complex structure of the resolution $\tilde{Z}$ of $(Z',\calA_1)$. The mapping
\[\rho':(\tilde{Z},\calC_1)\to (Z',\calA_1)\]
can be factored into $\rho_{min}\comp\rho'''$, where $\rho''':\tilde{Z}\to Z_{min}$ is a chain of blow ups at points.
We claim that there is a unique complex structure $\calC_2$ on $\tilde{Z}$ which makes 
\[\rho''':(\tilde{Z},\calC_2)\to (Z_{min},\calB_2)\]
analytic. Indeed, as $\rho'''$ is a composition of blowing ups at points, by induction, it is sufficient to consider the
case in which $\rho'''$ is a single blow up at a 
point $p$. The topological space obtained by blowing up a surface at a smooth point does not depend on the complex structure of the surface. Thus, the
complex structure on the blow up of $(Z_{min},\calB_2)$ at $p$ is the required complex structure at $\tilde{Z}$.

\textsc{Step 4}. The mapping
\[\rho':(\tilde{Z},\calC_2)\to (Z',\calA_2)\]
is analytic for being a composition of analytic mappings.

By Stein factorisation the composition
\[\pi\comp\eta\comp\tilde{\psi}'\comp\rho':(\tilde{Z},\calC_2)\to (X,O)\]
factors into $\phi\comp\sigma$ where
\[\phi:(Z_2,Q_2)\to (X,O)\]
is a finite analytic mapping of normal surface singularities, and 
\[\sigma:(\tilde{Z},\calC_2)\to (Z_2,Q_2)\]
is a resolution of singularities. From the topological viewpoint the mapping $\sigma$ consists in collapsing the exceptional divisor $F$ to a point.
Hence, the singularity $(Z_2,Q_2)$ is topologically equivalent to $(Z,Q)$: the germ $(Z_2,Q_2)$
may be viewed as the germ of topological space $(Z,Q)$ with a different analytic structure, and the mapping $\sigma$ coincides with $\rho$. By construction the mapping $\phi$ is equal to the composition $\nu\comp\psi$. 
This ends the proof of the Lemma.
\end{proof}

Denote by $\calD_1$ and $\calD_2$ the analytic structures of $(Z,Q)$ and $(Z_2,Q_2)$. We denote 
the germs $(Z,Q)$ and $(Z_2,Q_2)$ respectively by $(Z,Q,\calD_1)$ and $(Z,Q,\calD_2)$. The
mapping $\sigma$ coincides then with
\begin{equation}
\label{resrejdt}
\rho:(\tilde{Z},\calC_2)\to (Z,Q,\calD_2).
\end{equation}

\textbf{Part III}.
\begin{lema}
\label{comproblr}
The finite analytic morphism
$$\nu\comp\psi:(Z,Q,\calD_2)\to (X,O)$$
realises an adjacency from $E_u$ to $\gamma'$ and
avoids $\Delta_{E_u}$ if $\psi$ avoids $\Delta_{E_u}$.
\end{lema}
\begin{proof}
The resolution $\rho$ defined in~(\ref{resrejdt}) 
resolves the indeterminacy of $\pi^{-1}\comp\nu\comp\psi$, since, by construction, this rational map coincides with the
morphism 
$\eta\comp\tilde{\psi}=\eta\comp\tilde{\psi'}\comp\rho'$. The exceptional divisor of $\rho$ is $F$, and, since $\eta$ leaves each irreducible component of $E$
invariant, we have that $\eta\comp\tilde{\psi}(F_0)$ is a subset of $E_u$.

Let
\[\varphi:(\CC,O)\to (\tilde{Z},\calC_1)\]
be the arc predicted in Definition~\ref{branchadj}. Since $\tilde{\psi}'\comp\rho'\comp\varphi$ has image in an open set
in which $\eta$ is biholomorphic, the 
arc $\varphi$ is also holomorphic with respect to the complex structure of $\calC_2$.
By construction we have the equality
$\nu\comp\psi\comp\rho\comp\varphi=\gamma'$. Let
\[\rho_1:(\tilde{Z},\calC_1)\to (W,\calE_1)\]
be the sequence of contractions of $(-1)$-curves predicted in Definition~\ref{branchadj}. The same sequence of contractions in $(\tilde{Z},\calC_2)$ gives rise to
a unique new complex structure $\calE_2$ in $W$, which makes 
\[\rho_1:(\tilde{Z},\calC_2)\to (W,\calE_2)\]
holomorphic. By construction $\rho_1$ does not collapse $F_0$, the arc $\rho_1\comp\varphi$ sends the special point into a point $Q'$ of $F_0$ in which
$F_0$ has a smooth branch $L$ and moreover $\rho_1\comp\varphi$ is transverse to $L$ at $Q'$.
Thus the covering $\nu\comp\psi$ realises an adjacency from $E_u$ to $\gamma'$ (see Definition~\ref{branchadj}).

Let us prove now that $\psi$ avoids $\Delta_{E_u}$ if and only if $\nu\comp\psi$ also avoids it.
Let $\rho=\rho_2\comp\rho_1$ be the factorisation of $\rho$ predicted in Definition~\ref{branchadj}.
The situation of Definition~\ref{branchadj} in the complex analytic setting (Definition~\ref{branchadjC})
makes clear the existence of a small neighborhood $\Omega$ of $Q'$ in $W$ such that
the only point of indetermination of the meromorphic map $\overline{\psi}:=\pi^{-1}\comp\psi\comp\rho_2$ in $\Omega$ is $Q'$.

{\bf Claim:} we have that $\psi$ avoids $\Delta_{E_u}$ if and only if $\Omega$ can be taken small enough so that for any $Q''\in\Omega\setminus\{Q\}$ there exists a germ $(D,Q'')$ of smooth real $2$-dimensional subvariety of $W$ which is
transverse to the smooth branch $L$ at $Q'$ and such that the restriction of 
\[\overline{\psi}:(D,Q'')\to \tilde{X}\]
is transverse to $E_u$ at $Q''$ in the smooth category.

The claim implies that avoiding $\Delta_{E_u}$ only depends of the restriction of 
the lifting $\overline{\psi}$ to an open subset $\Omega'\subset W$ containing an small punctured neighborhood of 
$Q'$ in $L$. Since $\eta$ is a diffeomorphism the lifting $\overline{\nu\comp\psi}$ 
corresponding to $\nu\comp\psi$ satisfies the transverseity condition of the claim if and only if $\overline{\psi}$
satisfies it. Thus, if the claim is true we have that
$\psi$ avoids $\Delta_{E_u}$ if and only if $\nu\comp\psi$ does, and the proof of the lemma is complete.

Let us prove the claim.
Consider $(\CC^2,O)$ with local coordinates $(s,t)$. There is a local analytic diffeomorphism 
\begin{equation}
\label{cartalocal}
\beta:(\CC^2,O)\to (W,Q')
\end{equation}
such that $\beta(s,0)$ parametrises $(L,Q')$ and $\beta(0,t)$ equals $\rho_1\comp\varphi(t)$.

If $\psi$ avoids $\Delta_{E_u}$, by Definition~\ref{branchadjC} (that is Definition~\ref{branchadj} in the complex 
analytic setting), it is possible to chose $\beta$ in such a way that
for any $s_0\neq 0$ and small enough the arc $\psi\comp\rho_2\comp\beta(s_0,t)$ has a transverse lifting to $E_u$.
Since the lifting is equal to $\overline{\psi}\comp\beta(s_0,t)$ we have proved the ``only if'' part.

If there is a $2$-dimensional germ $(D,Q'')$ of smooth subvariety of $W$ which is 
transverse to the smooth branch $L$ at $Q'$ and such that the restriction of 
\[\overline{\psi}:(D,Q'')\to \tilde{X}\]
is transverse to $E_u$ at $Q''$ in the smooth category then an easy manipulation with tangent spaces and tangent maps
shows the transverseity to $E_u$ of the restriction of $\overline{\psi}$ to any other $2$-dimensional germ of smooth
subvariety of $W$ which is transverse to the smooth branch $L$ at $Q''$. 
This easily implies that for any holomorphic arc
\[\theta:(\CC,O)\to (W,Q'')\]
transverse to $L$ the lifting to $\tilde{X}$ of $\psi\comp\rho_2\comp\theta$ is transverse to $E_u$. This shows the 
``if'' part.
\end{proof}
The last Lemma completes the proof of Proposition~\ref{movinbc}.
\end{proof}

The same technique leads to the following:

\begin{prop}
\label{toptype}
Let $\kappa:\calG_1\to\calG_2$ be an isomorphism between the weighted graphs of the minimal good resolution of two surface singularities 
$(X_1,O)$ and $(X_2,O)$. Let $E_u$ and $E_v$ be two exceptional divisors of the minimal good resolution of $(X,O)$. If there is a finite analytic mapping realising
an adjacency from $E_u$ to $E_v$ (avoiding $\Delta_{E_u}$) then there is a finite analytic mapping realising an adjacency from $E_{\kappa(u)}$ to $E_{\kappa(v)}$ (avoiding $\Delta_{E_{\kappa(u)}}$) .
\end{prop}
\begin{proof}
Let $\gamma$ be a convergent $\CC$-arc in $\dot{N}_{E_v}$ such that there exists a finite analytic mapping 
\[\psi:(Z,Q)\to (X_1,O)\]
realising an adjacency from $E_u$ to $\gamma$. Let 
$\gamma'$ be any convergent $\CC$-arc in $\dot{N}_{E_{\kappa(v)}}$. Let 
\[\pi_i:\tilde{X}_i\to (X_i,O)\]
be the minimal good resolutions of $X_i$ for $i=1,2$. 
Let
\[\rho:\tilde{Z}\to (Z,Q)\]
be the resolution of singularities predicted in Definition~\ref{branchadj}, and
$\tilde{\psi}:\tilde{Z}\to\tilde{X}_1$ the lifting of $\psi$. By Stein factorisation $\tilde{\psi}$ factors into
$\tilde{\psi}'\comp\rho'$, where
\[\tilde{\psi}':Z'\to \tilde{X}_1\]
is a finite mapping (that is, a branched covering) with $Z'$ a normal surface.

Let $\Delta'$ be the union of the irreducible components of the branching locus of $\tilde{\psi}'$ not contained in the exceptional
divisor of $\pi_1$. Let $\tilde{\gamma}$ and $\tilde{\gamma}'$ be liftings of $\gamma$ and $\gamma'$ to $\tilde{X}_1$ and $\tilde{X}_2$.

Since $\tilde{X_1}$ and $\tilde{X_2}$ are diffeomorphic to the plumbing $4$-manifolds of the dual graph of the resolutions, which are isomorphic by the isomorphism $\kappa$,
possibly having to shrink them to a smaller tubular neighborhood 
of the exceptional divisor of $\pi_i$ we may assume:
\begin{enumerate}
\item For $i=1,2$ there are neighborhoods $U^i_1,...,U^i_m$ of each singular point of the exceptional divisor in
$\tilde{X}$ which, if $i=1$,
contain all the components of $\Delta'$ which meet the corresponding singular point.
\item For $i=1,2$ the difference $\tilde{X_i}\setminus \cup_{j=1}^mU_j$ splits in one connected component $\tilde{X_i}[v]$ for each
irreducible component $E^i_v$ of $\pi_i^{-1}(O)$. 
In addition, for any $v$ we have that $\hat{E}^i_v:=E^i_v\setminus\cup_{j=1}^mU^i_j$ is the
deletion of several discs in $E^i_v$, and we have a smooth
product structure $\tilde{X}^i[v]=\hat{E}^i_v\times D$, where $D$ is a disk.
\item Consider the set $\{p^1_1,...,p^1_s\}$ of meeting points of $\Delta'$ with $\pi_1^{-1}(O)$
which are not singularities of $\pi^{-1}(O)$. 
Given any point $p^1_i$ contained in a divisor $E^1_u$, there exists a disk $D^1_i$ around 
$p^1_i$ in $\hat{E}^1_u$ such that all the components of $\Delta'$ meeting $p^1_i$ are contained in $D^1_i\times D$,
and any component of $\Delta'$ meeting $D^1_i\times D$ meet $\pi^{-1}_1(O)$ at $p^1_i$.
In addition the closure of the discs $D^1_1$,...,$D^1_s$
are two-by-two disjoint and do not meet the boundary of $\hat{E}^1_v$.
\item Given any component $E^1_u$ of $\pi_1^{-1}(O)$ let $\{p^1_{i(1)},...,p^1_{i(l_u))}\}$
be the subset formed by the points in $\{p^1_1,...,p^1_s\}$ belonging to $E^1_u$. Choose points 
$$\{p^2_{i(1)},...,p^2_{i(l_u))}\}$$
 in $\hat{E}^2_\kappa(u)$ and discs $D^2_{i(1)}$,...,$D^1_{i(l_u)}$ 
with pairwise diferent closure not meeting the boundary of $\hat{E}^2_{\kappa(u)}$.
\item The images of $\tilde{\gamma}$ and $\tilde{\gamma}'$ are fibre disks in $\tilde{X^1}[v]$ and $\tilde{X^2}[{\kappa(v)}]$
by the projection to 
$\hat{E}^1_v$ and $\hat{E}^2_{\kappa(v)}$ respectively. Either there is a disk
$D_{\tilde{\gamma}}$ in $\hat{E}^1_v$ around $\tilde{\gamma}(0)$ with closure disjoint to the closure of any $D^1_i$ 
and to the boundary of $\hat{E}^1_v$, or $\tilde{\gamma}$ is the fibre of a point in $\{p^1_1,...,p^1_s\}$
by the projection to $\hat{E}^1_v$. In the later case we define the disk $D^1_{\tilde{\gamma}}$ to be the equal to $D^1_j$,
for $\tilde{\gamma}(0)=p^1_j$, we chose $p^2_j$ to be equal to $\tilde{\gamma}'(0)$, and define the 
disc $D^2_{\tilde{\gamma}}$ to be equal to $D^2_j$. Otherwise we define $D^2_{\tilde{\gamma}}$ to be a disc in 
$\hat{E}^2_\kappa(v)$ centered at $\tilde{\gamma}'(0)$ and with with closure disjoint to the closure of the $D^2_i$'s 
and disjoint to the boundary of $\hat{E}^2_{\kappa(v)}$.
\end{enumerate}

Choose biholomorphisms from $U^1_i$ to $U^2_i$ for any $i$, from $D^1_i\times D$ to $D^2_i\times D$, and from 
$D^1_{\tilde{\gamma}}\times D$ to $D^2_{\tilde{\gamma'}}\times D$ such that, using the methods of the proof of Lemma~\ref{difeolr} we may extend to a diffeomorphism
\[\eta:\tilde{X}_1\to\tilde{X_2}\]
with the following properties:
\begin{enumerate}[(i)]
\item We have $\eta(\pi_1^{-1}(0))=\pi_2^{-1}(0)$, $\eta(E_u)=E_{\kappa(u)}$, and $\eta(E_v)=E_{\kappa(v)}$.
\item the restriction of $\eta$ to a neighborhood of $\Delta'\cup Sing(\pi_1^{-1}(0))\cup \tilde{\gamma}(\CC,0)$ is biholomorphic.
\item we have $\eta\comp\tilde{\gamma}=\tilde{\gamma}'$. 
\end{enumerate}

As (by properties (i) and (ii)), the branching locus of $\eta\comp\psi$ is an analytic subset of $\tilde{X}_2$, as in the proof of Lemma~\ref{newcompstr}
we can change the complex structure in $Z'$ so that the mapping $\eta\comp\psi'$ becomes holomorphic. The construction of the branched cover realising 
an adjacency from $E_{\kappa(u)}$ to $\gamma'$ follows now word by word the analogous constructions in the proofs of 
Lemmata~\ref{newcompstr}~and~\ref{comproblr}. 
The fact that if the original finite analytic mapping avoids $\Delta_{E_u}$ then the constructed one avoids
$\Delta_{E_{\kappa(u)}}$ also follows like the analogous statement in Lemma~\ref{comproblr}.
\end{proof}

\section{On Nash and wedge problems}
\label{principal}

The following theorem contains Theorem A stated at the introduction.

\begin{theo}
\label{Nashwedge1}
Let $(X,O)$ be a normal surface singularity defined over a uncountable algebraically closed field $\KK$ of characteristic $0$.
Let $E_v$ be an essential irreducible component of the exceptional divisor of a resolution. Equivalent are:
\begin{enumerate}
\item The set $N_{E_v}$ is in the Zariski closure of $N_{E_u}$, with $E_u$ another component of the exceptional divisor.
\item Given any proper closed subset $\calZ\subset\overline{N}_{E_u}$ there exists an algebraic $\KK$-wedge realising an adjacency from $E_u$ to $E_v$ and avoiding $\calZ$.
\item There exists a formal $\KK$-wedge realising an adjacency from $E_u$ to $E_v$.
\item Given any proper closed subset $\calZ\subset\overline{N}_{E_u}$ there exists a finite morphism realising an adjacency from $E_u$ to $E_v$ and avoiding $\calZ$.
\end{enumerate}
If the base field is $\CC$ the following further conditions are equivalent to those above:
\begin{enumerate}[(a)]
\item Given {\em any} convergent arc $\gamma\in\dot{N}_{E_v}$ there exists a convergent $\CC$-wedge realising an adjacency from $E_u$ to $\gamma$ and avoiding $\Delta_{E_u}$
(that is, with generic arc tranverse to $E_u$ at an smooth point of $E$).
\item Given {\em any} convergent arc $\gamma\in\dot{N}_{E_v}$ there exists a convergent $\CC$-wedge realising an adjacency from $E_u$ to $\gamma$.
\item Given {\em any} convergent arc $\gamma\in\dot{N}_{E_v}$ there exists a finite morphism realising an adjacency from $E_u$ to $\gamma$ and avoiding $\Delta_{E_u}$.
\end{enumerate}
\end{theo}

A component $E_v$ is in the image of the Nash map if and only if $N_{E_u}$ is {\em not} in the Zariski closure of $N_{E_u}$ for a different irreducible component $E_u$
of the exceptional divisor. The negation of the statements of the previous theorem characterises the image of the Nash map in terms of wedges defined over the base field,
and in particular gives Corollary B stated at the introduction.

\begin{proof}[Proof of Theorem~\ref{Nashwedge1}]

By Proposition~3.8~of~\cite{Re} the set $\overline{N}_{E_v}$ is generically stable, and hence,
if $N_{E_v}$ is in the Zariski closure
of $N_{E_u}$ then, given any proper closed subset $\calZ$ of $\overline{N}_{E_u}$, by Lemma~\ref{curveselectionlemma} there exists a $K$-wedge whose special arc is the generic point of $N_{E_v}$,
and whose general arc belongs to $N_{E_u}\setminus\calZ$.
Thus, by Proposition~\ref{especializacion}, there exists a formal $\KK$-wedge realising an adjacency from $E_u$ to $E_v$ and avoiding $\calZ$. By Proposition~\ref{approximation} we find an algebraic $\KK$-wedge
with the same properties. Thus (1) implies (2). 

Obviously (2) implies (3). By Proposition~\ref{wedcov} we have that (2) implies (4) and that (4) implies (3). Proposition~\ref{approximation} together with Proposition~\ref{wedcov} shows that
(3) implies (4). To prove the first set of equivalences we only need to show that (4) implies (1).

Let
\begin{equation}
\label{e0}
\psi:(Z,Q)\to (X,O)
\end{equation}
be a finite morphism realising an adjacency from $E_u$ to $E_v$. In order to finish the proof we have to show that $N_{E_v}$ is in the Zariski closure of $N_{E_u}$.
It is well known that any morphism of algebraic varieties over a field of characteristic $0$ is defined over a subfield $\KK_1\subset\KK$ which is a finitely generated
extension of $\QQ$ (just observe that there are only finitely many coefficients involved in the equations of the varieties and of the morphism).
Explicitly this means that there is a finite
morphism
\begin{equation}
\label{e1}
\psi':(Z',Q')\to (X',O')
\end{equation}
defined over $\KK_1$ such that the morphism~(\ref{e0}) is obtained by pull-back of~(\ref{e1}) by
\begin{equation}
\label{basechange1}
Spec(\KK)\to Spec(\KK_1).
\end{equation}

Let $\gamma$ be the algebraic $\KK$-arc such that $\psi$ realises an adjacency from $E_u$ to $\gamma$. The algebraicity of $\gamma$ implies that there are polynomials
$F_1,...,F_N\in\KK[t,x_1,...,x_n]$ such that we have the equalities
\[F_i(t,x_i\comp\gamma(t))=0\]
in $\KK[[t]]$. Let $\KK_2$ be the finitely generated extension of $\KK_1$ obtained adjoining the coefficients of the polynomials $F_1,...,F_N$. Let $\overline{\KK}_2$ be the
algebraic closure of $\KK_2$. For any $i\leq N$ we define the algebraic curve $C_i\subset\overline{\KK}^2_2$ by the equation 
\[F_i(t,x)=0.\]
The formal arc $(t,x_i\comp\gamma(t))$ parametrises a smooth formal branch $L$ of $C_i$ at $(0,0)$. Thus, the polynomial $F_i$ admits a factor 
$G_i\in\overline{\KK}_2[[t,x_i]]$ whose linear part is of the form $at+bx_i$ with $b\neq 0$ such that
\[G_i(t,x_i\comp\gamma(t))=0.\]
The last equation and the form of the linear part allows to reconstruct the coefficients of $x_i\comp\gamma(t)$ and show that they belong to $\overline{\KK}_2$.
We have proved that $\gamma$ is a $\overline{\KK}_2$-arc.

Since $\KK_2$ is finitely generated over $\QQ$ it admits an embedding in $\CC$, and also its algebraic closure $\overline{\KK}_2$. We redefine $\KK_1$ to be equal to
the algebraic closure $\overline{\KK}_2$. Let
\begin{equation}
\label{res'}
\pi':\tilde{X}'\to X'
\end{equation}
be the minimally good resolution of the singularity of $X'$ at $O'$. Let 
\[E':=\cup_uE'_u\]
be a decomposition of the exceptional divisor in irreducible components. The resolution $\pi:\tilde{X}\to X$ and the divisor $E_u$ are the respective pull-backs of $\pi'$ and $E'_u$
by~(\ref{basechange1}). It is clear that the arc $\gamma$ belongs to $\dot{N}_{E_v}$ as $\KK$-arc if and only if it belongs to $\dot{N}_{E'_v}$ as $\KK_1$-arc.

On the other hand, since $E_u$ is the base change of $E'_u$ under~(\ref{basechange1}) it is cleat that $N_{E_u}$ is the base change of $N_{E'_u}$
under ~(\ref{basechange1}). Thus $N_{E_v}$ is in the Zariski closure of $N_{E_u}$ if and only if $N_{E'_v}$ is in the Zariski closure of $N_{E'_u}$.

Since $\KK_1$ is the algebraic closure of a finitely generated extension of $\QQ$, it admits an embedding into $\CC$ which gives rise to a morphism
\begin{equation}
\label{basechange2}
Spec(\CC)\to Spec(\KK_1).
\end{equation}
Let
\begin{equation}
\label{e2}
\psi'':(Z'',Q'')\to (X'',O'')
\end{equation}
and
\begin{equation}
\label{res''}
\pi'':\tilde{X}''\to X''
\end{equation}
be the base change of~(\ref{e1}) and~(\ref{res'}) by~(\ref{basechange2}). Let
\[E'':=\bigcup_u E''_u\]
be a decomposition of the exceptional divisor of $\pi''$ in irreducible components. 
As before $N_{E''_v}$ is in the Zariski closure of $N_{E''_u}$ if and only if $N_{E'_v}$ is in the Zariski closure of $N_{E'_u}$.

The mapping~(\ref{e2}) is a finite morphism defined over $\CC$ defining an adjacency from $E''_u$ to the algebraic arc $\gamma$, which, viewed as a $\CC$-arc belongs to
$\dot{N}_{E''_v}$. By Proposition~\ref{wedcov}~(1) there exists a $\CC$-wedge realising an adjacency from $E''_u$ to $\gamma$. Hence, by Proposition~\ref{movinw}, for any 
convergent $\CC$ arc $\gamma'$ in $\dot{N}_{E''_v}$ there exists a $\CC$-wedge realising an adjacency from $E_u$ to $\gamma'$.

A $\CC$-wedge realising an adjacency from $E_u$ to $\gamma'$ corresponds to a $\CC$-arc in $\calX_\infty$, taking the generic point into $N_{E''_u}$ and the special point 
$\gamma'$. Hence every convergent $\gamma'$ in $\dot{N}_{E''_v}$ belongs to the Zariski-closure of $N_{E''_u}$. As, by Artin's Approximation Theorem, any formal arc admits
a convergent approximation coinciding with the original one up to any fixed order, the set of convergent $\CC$-arcs in $\dot{N}_{E''_v}$ is Zariski dense
in $\dot{N}_{E''_v}$. As $\dot{N}_{E''_v}$ is Zariski dense in $\overline{N}_{E''_v}$ we have that $N_{E''_v}$ is in the Zariski closure of $N_{E''_u}$.

Assume now that the base field is $\CC$.
Proposition~\ref{movinw} shows that (2) imples (a), which clearly implies (b). Since the set of convergent transverse arcs theough $E_v$ is dense in $N_{E_v}$ we have that (b) imples (1).
The conditions (c) and (d) are also equivalent by a similar reasoning replacing Proposition~\ref{movinw} by Proposition~\ref{movinbc}.
\end{proof}

\begin{theo}
\label{combinatorio}
Let $\kappa:\calG_1\to\calG_2$ be an isomorphism between the weighted graphs of the minimal good resolution of two normal surface singularities $(X_1,O)$ and $(X_2,O)$ 
defined over uncountable algebraically closed fields $\KK_1$ and $\KK_2$.
Let $E_u$ and $E_v$ be two exceptional divisors of the minimal good resolution of $(X_1,O)$.
\begin{enumerate}
\item There is a finite analytic mapping realising
an adjacency from $E_u$ to $E_v$ if and only if there is a finite analytic mapping realising an adjacency from $E_{\kappa(u)}$ to $E_{\kappa(v)}$.
\item There is a $\KK_1$-wedge realising
an adjacency from $E_u$ to $E_v$ if and only if there is a $\KK_2$-wedge realising an adjacency from $E_{\kappa(u)}$ to $E_{\kappa(v)}$.
\item The set $N_{E_v}$ is in the Zariski closure of $N_{E_u}$ if and only if the set $N_{E_{\kappa(v)}}$ is in the Zariski closure of $N_{E_{\kappa(u)}}$.
\end{enumerate}
\end{theo}
\begin{proof}
By Theorem~\ref{Nashwedge1} the three statements are equivalent. As in the proof of Theorem~\ref{Nashwedge1} we reduce the first statement to the case in which 
$\KK_1=\KK_2=\CC$. Then the Theorem follows from Proposition~\ref{toptype}.
\end{proof}

\begin{definition}
\label{adjgraph}
Let $(X,O)$ be a normal surface singularity defined over an algebraically closed field $\KK$.
The arc-adjacency graph of $(X,O)$ for a resolution is the directed graph whose vertices
correspond in a bijective manner with the irreducible components of the exceptional divisor of the resolution of $(X,O)$ and has an arrow from the vertex corresponding to
$E_u$ to the vertex corresponding to $E_v$ if and only if $N_{E_v}$ is in the Zariski closure of $N_{E_u}$.
An arrow is called {\em trivial} if a sequence of contractions of rational curves with self-intersection $-1$ collapses $E_v$ into $E_u$ (in this case it is clear that 
 $N_{E_v}$ is in the Zariski closure of $N_{E_u}$).
\end{definition}

\begin{remark}
A set $N_{E_v}$ is in the image of the Nash mapping if and only if its vertex in the adjacency graph only has incoming trivial arrows.
The Nash mapping is bijective if and only if the adjacency graph contains only trivial arrows.
\end{remark}

The previous Theorem may be read as follows:

\begin{cor}
\label{topo}
Let $(X,O)$ be a normal surface singularity defined over an algebraically closed field $\KK$ of characteristic $0$ (without the uncountability hypothesis). 
The minimal resolution graph of $(X,O)$ determines the adjacency graph of any resolution. In the case of complex analytic singularities
the topology of the abstract link determines the adjacency graph of any resolution. Hence the bijectivity of the Nash mapping is a topological property of the singularity.
\end{cor}
\begin{proof}
The result is clear if $\KK$ is uncountable. For the countable case it is enough to observe that the resolution and adjacency graphs are preserved by base change.
\end{proof}

\subsection{The problem of lifting wedges}
\label{wedgesclasico}
In the rest of the section we link our results with the problem of lifting wedges as studied in~\cite{Le},~\cite{Re}, and~\cite{LR}.
Here we work over the complex numbers. We remind some terminology:

A $K$-wedge $\alpha$ lifts to $\tilde{X}$ if the rational map $\pi^{-1}\comp\alpha$ is a morphism. A $K$-wedge is centred at the generic point of
$\overline{N}_{E_i}$ if its special arc is the generic point of $\overline{N}_{E_i}$. 
The resolution $\pi:\tilde{X}\to X$ satisfies the property of lifting wedges centred at the generic point of $\overline{N}_{E_v}$ if any such wedge lifts to $\tilde{X}$.
By Theorem~5.1~of~\cite{Re} this equivalent to the fact that $E_v$ is in the image of the Nash map.

In Proposition~2.9~of~\cite{LR} a sufficient condition for a resolution to have the property of lifting wedges centred at the generic point of $\overline{N}_{E_v}$ 
is given in the following way: it is sufficient to check that any $\KK$-wedge whose special arc is transverse to $E_i$ arc through a very dense collection of closed 
points of $E_i$ lifts to $\tilde{X}$ (a very dense set is a set which intersects any countable intersection of dense open subsets). We improve this result 
for normal complex surface singularities in the following
Theorem by proving that it is sufficient that there exists a single, convergent 
transverse arc to $E_v$ such that any $\CC$-wedge having it as special arc lifts.
A use of Lefschetz Principle as above shows that the assumption that the base field is $\CC$ is harmless.

\begin{cor}
\label{liftgeneric}
Let $(X,O)$ be a normal surface singularity defined over $\CC$. Let 
\[\pi:\tilde{X}\to X\]
be a resolution of singularities and $E_v$ any essential irreducible component of the exceptional divisor.
If there exists a convergent arc $\gamma\in\dot{N}_{E_v}$ such that
any $\KK$-wedge having $\gamma$ as special arc lifts to $\tilde{X}$, then any resolution of $\tilde{X}$ has the property of lifting wedges centred at the generic point 
of $\overline{N}_{E_v}$ (or equivalently, the component $E_v$ is in the image of the Nash map).
\end{cor}
\begin{proof}
By Theorem~\ref{Nashwedge1} the component $E_v$ is in the image of the Nash map if and only if it does not exist any formal $\CC$-wedge realising an adjacency from
$E_u$ to $E_v$, for $E_u$ a component of the exceptional divisor of $\pi$ different than $E_v$. If such a $\CC$-wedge exists, by Proposition~\ref{movinw}
there exists a $\CC$-wedge realising an adjacency from $E_u$ to $\gamma$. Since such a wedge has $\gamma$ as special arc and does not lift to $\tilde{X}$ we obtain a 
contradiction.
\end{proof}

\section{Applications}
\label{aplicaciones}

In what follows we establish a few relations between the resolution graph and the adjacency graph.
As both graphs are combinatorial we can work over $\CC$ in the proofs without loosing generality.

We start with a result for graphs with symmetries:

\begin{prop}
\label{simmetries}
If $\varphi$ is an automorphism of a graph $\calG$ not fixing an essential vertex $u$ then there is no adjacency from $u$ to $\varphi(u)$.
\end{prop}
\begin{proof}
Let $\{u=u_0,...,u_{d-1}\}$ be the orbit of $u$ by $\varphi$ (we order it so that $u_{i+1}=\varphi(u_{i})$ with the indexes taken 
modulo $d$).
Let
\[\pi:\tilde{X}\to (X,O)\]
be a resolution of a normal surface singularity with graph $\calG$. Let $E_{u_i}$ be the irreducible component of the exceptional divisor corresponding to the vertex $u_i$. As essentiality is a combinatorial
property of the graph, since $E_{u_1}$ is essential, each $E_{u_i}$ is essential.

Suppose we have an adjacency from $u_0=u$ to $u_1=\varphi(u)$. This means that we have the inclusion 
\[\overline{N}_{E_{u_0}}\supset\overline{N}_{E_{u_1}}.\]
By Theorem~\ref{combinatorio} we have an adjacency from $u_i$ to $u_{i+1}$ for any $i$ taking the indexes modulo $d$, and, hence, we have the chain of inclusions
\[\overline{N}_{E_{u_1}}\supset\overline{N}_{E_{u_2}}\supset ...\supset\overline{N}_{E_{u_0}}.\]
We conclude the equality $\overline{N}_{E_{u_0}}=\overline{N}_{E_{u_1}}$. 

On the other hand, since $E_{u_0}$ and $E_1$ are essential, and by Corollary~3.9~of~\cite{Re}, the sets $\dot{N}_{E_{u_0}}$ and $\dot{N}_{E_{u_1}}$ are Zariski open in the 
irreducible set $\overline{N}_{E_{u_0}}=\overline{N}_{E_{u_1}}$. Since $\dot{N}_{E_{u_0}}$ and $\dot{N}_{E_{u_1}}$ are disjoint we obtain a contradiction.
\end{proof}

As a curiosity we obtain an affirmative answer for Nash problem for very symmetric graphs:

\begin{cor}

If the automorphism group of a graph $\calG$ acts transitively, then its adjacency graph contains only trivial arrows. In other words, Nash mapping is bijective for singularities having this kind of resolution graph.
\end{cor}
\label{simmetries2}

The following Theorem is related to results of C.~Plenat (see~\cite{Pl}, Proposition~3.1 and Corollary~3.4).

\begin{cor}
\label{comparacion1}
The following statements hold:
\begin{itemize}
\item Let $\calG_1$ be a graph contained in a negative definite weighted graph $\calG_2$. 
If the adjacency graph of $\calG_1$ has an arrow from $u$ to $v$ then the adjacency graph of $\calG_2$ has also
an arrow from $u$ to $v$.
\item Let $\calG_2$ be a graph obtained from a negative definite weighted $\calG_1$ by decreasing the self-intersection weights of vertices (the graph $\calG_2$ is automatically negative definite).
If the adjacency graph of $\calG_2$ has an arrow from $u$ to $v$ then the adjacency graph of $\calG_1$ has also an arrow from $u$ to $v$.
\end{itemize}
\end{cor}
\begin{proof}
Let $(X_1,O_1)$ and $(X_2,O_2)$ be normal surface singularities whose resolutions 
\[\pi_i:\tilde{X}_i\to (X_i,O_i)\]
have graphs $\calG_1$ and $\calG_2$ respectively.
Taking a small neighborhood of the exceptional divisor the manifold $\tilde{X_1}$
is obtained by plumbing according to the graph $\calG_1$ of one bundle over $\PP^1$ with
Euler class the weight of the vertex.
Since the graph $\calG_2$ contains $\calG_1$ the smooth manifold $\tilde{X}_2$ may be obtained from $\tilde{X}_1$ by adding one bundle over $\PP^1$ for each vertex of $\calG_2\setminus\calG_1$
and one plumbing identification for each vertex $e$ of $\calG_2\setminus\calG_1$. Therefore we have a smooth local diffeomorphism
\[\varphi:\tilde{X}_1\to\tilde{X}_2.\]

We choose a complex structure in $\tilde{X}_1$ such that $\varphi$ becomes holomorphic, and change the complex structure of $(X_1,O_1)$ accordingly.

By the hypothesis, Theorem~\ref{Nashwedge1}~and Corollary~\ref{topo}, there is a $\CC$-wedge in $X_1$ realising an adjacency from $E_u$ to $E_v$.
By Proposition~\ref{approximation} there exists a convergent 
$\CC$-wedge
\[\alpha:(\CC^2,O)\to X_1\]
realising the same adjacency.
Let 
\[\sigma:W\to(\CC^2,O)\]
be the chain of blow ups at points giving the minimal resolution of indeterminacy of $\pi_1^{-1}\comp\alpha$, and 
\[\tilde{\alpha}:W\to\tilde{X}_1\] 
the lifting of $\alpha$. The mapping
\[\beta:=\pi_2\comp\varphi\comp\tilde{\alpha}\comp\sigma^{-1}:(\CC^2,O)\to X_2\]
is a well defined analytic mapping and a $\CC$-wedge in $X_2$ realising an adjacency from $E_u$ to $E_v$. This proves (1).

Now we deduce (2) from (1). Construct a graph $\calG_3$ as follows: for each vertex $v$ of $\calG_2$ let $a_v$ be the
difference of the self-intersection weight of the vertex $v$ in $\calG_1$ and the self-intersection weight of the
vertex $v$ in $\calG_1$. Add $a_v$ rational vertices with self-intersection weight equal to $-1$ to $\calG_2$,
attaching them to $v$. The graph $\calG_3$ is constructed after repeating the procedure for every vertex of $\calG_2$.
The graph $\calG_1$ is obtained from $\calG_3$ by deleting the vertices of $\calG_3\setminus\calG_2$, together with 
their attaching edges, and increasing in one unit for each deleted vertex the weight of the vertex of $\calG_2$ to 
which it was attached (this is the combinatorial counterpart of the blow-down operation). Since the adjacency between
components does not depend on the chosen resolution of singularities, and the second operation corresponds to collapsing $(-1)$-smooth rational curves, given $u,v\in\calG_1$ there is an arrow from $u$ to $v$ in the arc-adjacency
graph of $\calG_1$ if and only if there is an arrow from $u$ to $v$ in the arc-adjacency
graph of $\calG_3$. As $\calG_2$ is a subgraph of $\calG_3$ (2) follows from (1).
\end{proof}

Our aim now is to improve the following result of M. Lejeune-Jalabert and A. Reguera~\cite{LR}. 

\begin{prop}[M. Lejeune-Jalabert, A. Reguera]
\label{noracionales}
The following assertions hold:
\begin{enumerate}
\item If there is an arrow from a vertex $u$ to a vertex $v$ in the arc-adjacency graph of $(X,O)$ then there is a path joining the vertices $u$ and $v$ in the minimal resolution
graph such that all the vertices appearing in the path, with the possible exception of $u$, correspond to rational irreducible components of the exceptional divisor.
In particular, if $E_v$ is not rational, it is in the image of the Nash map.
\item If the Nash mapping is bijective for all graphs containing only rational curves, then it is bijective in general.
\end{enumerate}
\end{prop}

From now on we only work with singularities having resolutions such that all the irreducible components of the exceptional divisor are rational.

Let $\calG$ be a finite weighted graph with no edges starting and ending at the same point, and such that there not 
exist two different edges connecting the same pair of points. It is clear that any normal surface singularity has a
resolution whose associated graph has these properties.

The realisation $Z(\calG)$ of $\calG$ 
is the topological space constructed as follows: take a disjoint union of $3$-dimensional balls of radius $1$
indexed over the vertices of $\calG$ and copies of the interval $[0,1]$ indexed over the edges of $\calG$. For each edge of $\calG$, identify the ends of the 
corresponding interval with points of the boundary of the balls corresponding to the vertices which are joined by the edge. Do the identifications such that no point in
the boundary of a ball is the identification point of two different edges. It is clear that the topological type of $Z(\calG)$ only depends on $\calG$. The realisation
$Z(\calG)$ has the homotopy type of a wedge of circumferences. Given a map of graphs $\alpha:\calG_1\to\calG_2$ mapping vertices to vertices and edges to edges and respecting the incidence of the edges, its realisation is a continuous map 
\[t(\alpha):Z(\calG)\to Z'(\calG)\]
mapping the ball of a vertex $v$ homeomorphically to the ball of the vertex $\alpha(v)$ and the interval of and edge $e$ homeomorphically to the interval of the edge
$\alpha(e')$. The topological conjugation class of realisation of a mapping of graphs only depends on the mapping of graphs.

A loop in $\calG$ is a sequence of vertices and edges $\{v_1,e_1,v_2,e_2...,v_r,e_r\}$ such that $v_i$ is joined by $e_i$ to $v_{i+1}$ for any $i<r$ and 
$v_r$ is joined to $v_0$ by $e_r$. Notice that a loop is determined by its set of edges $\{e_1,...,e_r\}$.
A loop in $\calG$ is simple if there are no vertices repeated in the sequence. The finiteness of $\calG$ implies that there are finitely many simple loops $l_1,...,l_k$.
Since there are no edges starting and ending at the same point each simple loop contains at least two vertices. For each simple loop $l_i$ there is a simple 
closed curve $\beta_i$ in $Z(\calG)$ (in other words, an embedding of $\SSS^1$ into $Z(\calG)$) going, through the balls and intervals corresponding to the vertices and
edges $l_i$, in the order prescribed by $l_i$. The free homotopy class of $\beta_i$ in $Z(\calG)$ is determined by $l_i$,
and the homology class of $[\beta_i]\in H_1(Z(\calG),\ZZ)$ determines $l_i$. A trivial loop is a loop inducing the trivial element in the fundamental group.

We say that a loop $l=\{e_1,e_2...,e_r\}$ contains a simple loop $l_i:=\{f_1,f_2...,f_s\}$ if:
\begin{enumerate}
\item There is an increasing function
\[\theta:\{1,...,s\}\to\{1,...,r\}\]
such that $f_i=e_{\alpha(i)}$ and for any $i\leq s$ the sub-loop 
\[\{e_{\alpha(i)},e_{\alpha(i)+1},...,e_{\alpha(i+1)}\}\] 
is trivial (for the case $i=s$ consider the indexes modulo $s$).
\item The loop $\{e_{\alpha(1)},...,e_{\alpha(r)}\}$ is not a sub-loop of a trivial sub-loop of $l$.
\end{enumerate}

Given any finite covering $\varphi:Z'\to Z(\calG)$ we associate a finite graph $\calG(Z')$ to $Z'$ as follows: a point of $Z'$ is $1$-dimensional if it has
a neighborhood of topological dimension equal to $1$. Take a vertex $v'$ for each connected component $C_{v'}$ of the complement in $Z'$ 
of the set of $1$-dimensional points (intuitively take a vertex for each ``$3$-dimensional ball''). The component $C_{v'}$ is mapped under $\varphi$ to a $3$-dimensional ball corresponding to a vertex $v$. Give to $v'$ the weight of $v$. Join two vertices 
with an edge if they are connected in $Z'$ with an path homeomorphic to $[0,1]$ such that all the points in $(0,1)$ are $1$-dimensional. It is clear that $Z'$ is 
homeomorphic to the realisation $Z(\calG(Z'))$, and that there is a map 
\begin{equation}
\label{realimap}
\alpha:\calG(Z')\to\calG
\end{equation}
whose realisation is topologically conjugate to $\varphi$. A map of graphs is a finite covering if its realisation is a finite covering.

\begin{lema}
\label{indice}
For any $M\in\NN$ there exists a finite covering $\varphi:Z'\to Z(\calG)$ such that any non-trivial loop in $\calG(Z')$ has at least $M$ vertices.
\end{lema}
\begin{proof}
Let $\{l_1,...,l_k\}$ be the simple loops of $\calG$.

Consider any simple loop $l_i=\{v_1,e_1,...,v_r,e_r\}$ for $i\leq k$. 
Take $M$ disjoint copies $\calG^1,...,\calG^M$ of $\calG$. Let $l_i^j=\{v^j_1,e^j_1,...,v^j_r,e^j_r\}$ the 
loop corresponding to $l_i$ in $\calG^j$. In the disjoint union 
\[\coprod_{j=1}^r\calG^j\]
we change the attaching points of the edges $e^j_r$ for any $j$ as follows: the edge $e^j_r$ joins now the vertices $v^j_r$ and $v^{j+1}_1$ if $j<M$ and 
$e^M_r$ joins $v^m_r$ and $v^1_1$. The new graph is connected and denoted by $\calG_i$
Let
\[\alpha_i:\calG_i\to\calG\]
be defined as $\alpha(v^i_j):=v_j$ and $\alpha(e^i_j)=e_j$. The realisation
\[t(\alpha_i):Z(\calG_i)\to Z(\calG)\]
is a covering of degree $M$.
Let $l:=\{v'_1,e'_1,...,v'_r,e'_r\}$ be any loop in $\calG_i$. 
By construction, if its image under $\alpha_i$ contains the simple loop $l_i$ then it contains $M$ copies of it.

For any $i$ the projection to the $i$-th factor
\[pr_i:Z':=Z(\calG_1)\times_{Z(\calG)}...\times_{Z(\calG)} Z(\calG_k)\]
of the fibre product of the coverings $t(\alpha_i)$ is a covering.
Define 
\[\varphi:Z'\to Z(\calG)\]
by $\varphi:=\alpha_1\comp pr_1=...=\alpha_k\comp pr_k$.
Let $\alpha$ and $\beta_i$ be the mappings from $\calG(Z')$ to $\calG$ and $\calG_i$ whose respective realisations are $\varphi$ and $pr_i$ for any $i\leq k$.

We claim that the finite covering
\[\alpha:\calG(Z')\to\calG\] is the one we need.

Indeed,
let $l:=\{v'_1,e'_1,...,v'_r,e'_r\}$ be any non-trivial loop in $\calG(Z')$. Its image $\alpha(l)$ is a non-trivial loop in $\calG$, since a covering induces an
injection of fundamental groups. Thus $\alpha(l)$ contains a simple loop $l_i$. As $\alpha(l)$ equals $\alpha_i(\beta_i(l))$ and $\beta_i(l)$ is a non-trivial loop
in $\calG_i$ whose image under $\alpha_i$ contains $l_i$,
the loop $\alpha(l)$ contains $M$-copies of $l_i$, and hence the loop $l$ has at least $M$-vertices.
\end{proof}

\begin{prop}
\label{noloops}
Let $\calG$ be a negative definite
weighted graph with rational vertices. If there is an adjacency from a vertex $u$ to a vertex $v$ then there exists a negative-definite weighted graph $\calG'$ with rational vertices
and without loops, two vertices $u'$ and $v'$, with an adjacency from the first to the second, and a mapping of graphs $\alpha:\calG'\to\calG$ such that
$\alpha(u')=u$ and $\alpha(v')=v$. In addition $\alpha$ can be factorised as $\alpha=\beta\comp\iota$ where $\iota$ is the inclusion of $\calG'$ in a graph $\calG''$ and
$\beta$ is a finite covering from $\calG''$ to $\calG$.
\end{prop}
\begin{proof}
Let
\[\pi:\tilde{X}\to (X,O)\]
be a resolution of a complex analytic normal surface singularity with graph equal to $\calG$. Let $\pi^{-1}(0)=E=\cup_{u=1}^r E_u$ be a decomposition in irreducible
components of the exceptional divisor. There is a neighborhood of $O$ in $X$ such that $\tilde{X}$ admits a deformation retract
\[r:\tilde{X}\to E.\] We choose $X$ to be equal to this neighborhood.

By Theorem~\ref{Nashwedge1}~and Proposition~\ref{approximation} there exists a convergent $\CC$-wedge
\[\alpha:(\CC^2,O)\to X\]
realising an adjacency from $E_u$ to $E_v$. Let
\[\sigma:W\to\CC^2\]
be the composition of blow-ups at points giving the minimal resolution of indeterminacy of $\pi^{-1}\comp\alpha$. Let
\[\tilde{\alpha}:W\to\tilde{X}\]
be the lifting of $\alpha$. Let $K$ be the number of irreducible components of the exceptional divisor $F$ of $\sigma$.

By Lemma~\ref{indice} there exists a finite covering of graphs
\[\beta:\calG''\to\calG\] 
such that any loop in $\calG''$ has at least $K+1$-vertices. Now we will construct from it a finite morphism of 
singularities.

We construct a topological space $B$ as follows: take the disjoint union of a copy $B_i$ of $\PP^1$ for each irreducible component of $E_i$ of $E$, together with a 
homeomorphism $g_i:B_i\to E_i$. If two 
components $E_i$ and $E_j$ meet at at a common point $x$ in $E$ then add a copy of the interval $[0,1]$, and identify the point $0$ with $g_i^{-1}(x)$ and $1$ with 
$g_j^{-1}(x)$. The mapping
\[c:E'\to E\]
which contracts the copies of intervals to a point and whose restriction to $B_i$ coincides with $g_i$ is a homotopy equivalence.
Moreover the space $B$ is naturally included in the realisation $Z(\calG)$ (let $E'_i$ be the boundary of the $3$-ball corresponding to $E_i$), and the inclusion
\[\iota:B\to Z(\calG)\]
induces an isomorphism of fundamental groups. Therefore we have an isomorphism of fundamental groups
\[\iota_*\comp (c_*)^{-1}\comp r_*:\pi_1(\tilde{X})\to\pi_1(Z(\calG)).\]

The covering $\beta$ induces an injection of fundamental groups
\[t(\beta)_*:\pi_1(Z(\calG''))\to\pi_1(Z(\calG)).\]
Let
\[\tilde{\phi}:\tilde{X''}\to\tilde{X}\]
be the finite covering corresponding to the subgroup $(\iota_*\comp (c_*)^{-1}\comp r_*)^{-1}(\pi_1(Z(\calG''))$. We 
give to $\tilde{X''}$ the unique complex structure making $\tilde{\phi}$ holomorphic. It is easy to check that $\tilde{X''}$ is
the plumbing manifold associated to the graph $\calG''$. As a consequence $E'':=\tilde{\phi}^{-1}(E)$ is connected.

By Stein Factorisation Theorem there exists
a proper modification
\[\rho:\tilde{X''}\to X''\]
and a finite morphism
\[\phi:X''\to X,\]
with $X''$ normal such that $\phi\comp\rho=\pi\comp\tilde{\rho}$. As $(\pi\comp\tilde{\rho})^{-1}(O)=E''$ is connected,
there is a unique point $O''$ mapped to $O$ by $\phi$. Therefore $\phi$ is a finite analytic mapping of normal surface
singularities and the graph $\calG''$ is negative definite.

Since $W$ is simply connected there exists a lifting
\[\tilde{\alpha''}:W\to\tilde{X''}\]
which is holomorphic. Let $E'$ be the image of $F$ by $\tilde{\alpha''}$. Since $F$ has $K$ irreducible components, the
set $E'$ is a connected union of at most $K$ irreducible components. Let $\calG'$ be the subgraph of $\calG''$ 
corresponding to these components. Since each simple loops of $\calG''$ has at least $K+1$ vertices, the graph $\calG'$
has no loops. It is $\calG'$ is negative definite for being a subgraph of $\calG''$. Thus, by Grauert's Contraction Theorem~\cite{Gr}, there exists a birational morphism
\[\kappa:\tilde{X'}\to (X',O')\]
which contracts $E'$ giving rise to a normal surface singularity. The mapping
\[\kappa\comp\tilde{\alpha''}:(\CC^2,O)\to (X',O')\]
defines a wedge. Let $u'$ and $v'$ be the vertices of $G'$ such that the general arc of the wedge sends the special 
point into $E'_{u'}$ and the special arc rends the special point into $E'_{v'}$. Then $\kappa\comp\tilde{\alpha''}$
realises an adjacency from $E'_{u'}$ to $E'_{v'}$.
\end{proof}

Let us finish giving a name to the class of graphs to which we have reduced Nash problem.

\begin{definition}
\label{extremal}
A negative definite weighted graph is {\em extremal} if an only if it only has rational vertices, it has no loops, and if we increase the weight of any vertex the resulting graph is either not
negative definite, or the number of trivial arrows in the adjacency graph increases.
A $\QQ$-homology sphere is {\em extremal} if it is the boundary of the plumbing $4$-manifold of an extremal graph.
\end{definition}

\begin{cor}
If the Nash mapping is bijective for all singularities with resolution graph is extremal
then it is bijective in general. Equivalently, if the Nash mapping is bijective for all complex analytic normal
surface singularities having extremal $\QQ$-homology sphere links then it is bijective in general.
\end{cor}
\begin{proof}
Proposition~\ref{noracionales} reduces the the class of singularities having only rational exceptional divisors.
Proposition~\ref{noloops} allows to reduce to the class of 
singularities whose exceptional divisor is a tree of rational curves: an easy combinatorial argument shows that if $\varphi:\calG_1\to\calG_2$ is a finite covering 
of graphs a vertex $v\in\calG_1$ can be collapsed after a finite sequence of $-1$-vertices contractions if and only if and only if $\varphi(v)$ can be collapsed in 
the same way.

The reduction to the extremal case is made using Corollary~\ref{comparacion1}.
\end{proof}

\end{document}